\documentclass[11pt]{article}
\usepackage{amssymb,amsthm,amsfonts,amssymb,euscript,mathrsfs
}
\usepackage{amsxtra}
\usepackage{amscd}
\usepackage{latexsym}
\usepackage{exscale}
\usepackage[latin1]{inputenc}
\usepackage[cyr]{aeguill}
\usepackage[french]{babel}
\usepackage[T1]{fontenc}
\numberwithin{equation}{section}

\title{\bf Stabilité des sous-algèbres paraboliques de $\mathfrak{so}(n)$.}
\date{Université Tunis El-Manar\\
Faculté des sciences de Tunis\\
Département de Mathématiques\\
Campus universitaire\\
2092 El-Manar, Tunis, Tunisie\\
kais.amari@math.univ-poitiers.fr\\
ammarikais10@yahoo.fr}
\author{ AMMARI KAIS }

\makeatletter

\@addtoreset{equation}{section}
\makeatother

\newtheorem{theo}[subsubsection]{Th\'eor\`eme}
\newtheorem{pr}[subsubsection]{Proposition}
\newtheorem{lem}[subsubsection]{Lemme}
\newtheorem{defi}[subsubsection]{D\'efinition}

\newtheorem{rema}[subsubsection]{Remarque}
\newtheorem{ex}[subsubsection]{Exemples}
\newtheorem{cons}[subsubsection]{Conséquence}

\def\ind{\rm ind\,}
\def\rang{\rm rang\,}
\def\text{ \rm }
\def\em{ \rm }
\begin{document}
\maketitle
\textbf{Abstract}. Let $\mathbb{K}$ be an algebraically closed field of characteristic 0. A finite dimensional Lie algebra $\mathfrak{g}$ over $\mathbb{K}$ is said to be stable if there exists a linear form $g\in\mathfrak{g}^{*}$ and a Zariski
 open subset in $\mathfrak{g}^{*}$ containing $g$ in which all elements
 have their stabilizers conjugated under the connected adjoint
 group. It is well known that any quasi-reductive Lie algebra is stable. However, there are stable
 Lie algebras which are not quasi-reductive. This raises the question,
 if for some particular class of non-reductive Lie algebras, there is equivalence
 between stability and quasi-reductivity. In particular, it was conjectured in \cite[conjecture 5.6 (ii)] {Panyushev2005} that these two notions are equivalent for biparabolic subalgebras of a reductive Lie algebra. In this paper, we prove this conjecture for parabolic subalgebras of orthogonal Lie algebras and we answer positively to this question for certain Lie
 algebras which stabilize an alternating bilinear form of maximal rank
 and a flag in generic position.\\
\\\textbf{Résumé}. Soit $\mathbb{K}$ un corps algébriquement clos de caractéristique nulle. Une algèbre de Lie de dimension finie $\mathfrak{g}$ définie sur $\mathbb{K}$ est dite stable si elle possède une forme linéaire régulière $g \in \mathfrak{g}^{\ast}$ admettant un voisinage pour la topologie de Zariski dans lequel les stabilisateurs de deux éléments sont conjugués par le groupe adjoint connexe. Il est bien connu qu'une algèbre de Lie quasi-réductive est stable. Cependant, il existe des algèbres de Lie stables qui ne sont pas quasi-réductives. Se pose la question de savoir, si pour certaines classes particulières d'algèbres de Lie non réductives, il y a équivalence entre ces deux notions. En particulier, il a été conjecturé dans \cite[conjecture 5.6 (ii)] {Panyushev2005} que ces deux notions sont équivalentes pour les sous-algèbres biparaboliques d'une algèbre de Lie réductive. Dans cet article nous allons démontrer cette conjecture pour les sous-algèbres paraboliques des algèbres de Lie orthogonales et nous répondons positivement à cette question pour certaines algèbres de Lie qui stabilisent une forme bilinéaire alternée de rang maximal et un drapeau en position générique.
\begin{center} \section{ \textbf{Introduction} } \end{center}
\ \\
Dans toute la suite, $\mathbb{K}$ désigne un corps algébriquement clos de caractéristique nulle. Les algèbres de Lie considérées sont définies et de dimension finie sur $\mathbb{K}$. Soit $\mathfrak{g}$ une algèbre de Lie algébrique et $G$ un groupe de Lie algébrique affine connexe d'algèbre de Lie $\mathfrak{g}$. On munit $\mathfrak{g}^{\ast}$, l'espace dual de $\mathfrak{g}$, des actions coadjointes de $\mathfrak{g}$ et de $G$. \'{E}tant donnée une forme linéaire $g\in \mathfrak{g}^{\ast}$, on note $\mathfrak{g}(g)$ son stabilisateur dans $\mathfrak{g}$. On identifie $\mathfrak{g}(g)/\mathfrak{z}$, o\`{u} $\mathfrak{z}$ désigne le centre de $\mathfrak{g}$, avec son image dans $\mathfrak{g}\mathfrak{l}(\mathfrak{g})$.\\
\begin{defi} Une forme linéaire $g \in \mathfrak{g}^{\ast}$ est dite de type réductif si son stabilisateur dans $\mathfrak{g}$ pour la représentation coadjointe, modulo $\mathfrak{z}$, est une algèbre de Lie réductive dont le  centre est formé d'éléments semi-simples dans $\mathfrak{g}\mathfrak{l}(\mathfrak{g})$.
\end{defi}De manière équivalente, cela revient à demander que le groupe $G(g)/Z_{\mathfrak{g}}$ soit réductif, o\`{u} $G(g)$ désigne  le stabilisateur de $g$  dans $G$ et $Z_{\mathfrak{g}}$  le centralisateur de $\mathfrak{g}$ dans $G$.\\
\begin{defi}Une algèbre de Lie $\mathfrak{g}$ est dite quasi-réductive si elle possède  une forme linéaire de type réductif $g \in \mathfrak{g}^{\ast}$.
\end{defi}La notion de quasi-réductivité a été introduite par Duflo dans \cite{duflo-1982} (voir également \cite{DKT})  pour son importance dans la théorie des représentations.\\
 Beaucoup d'algèbres de Lie sont quasi-réductives. Par exemple, si $\mathfrak{g}$ est réductive, $\mathfrak{g}$ est quasi-réductive puisque $0 \in \mathfrak{g}^{\ast}$ est de type réductif. De plus,  les sous-algèbres paraboliques d'une algèbre de Lie simple $\mathfrak{s}$ de type $A$ et $C$ sont quasi-réductives  d'après un résultat de Panyushev (voir \cite{Panyushev2005}). Par contre, ceci n'est plus vrai pour les sous-algèbres paraboliques pour $\mathfrak{s}=\mathfrak{s}\mathfrak{o}(n,\mathbb{K})$ avec $n\geq 7$. Récemment, Duflo, Khalgui et Torasso ont donné une caractérisation des sous-algèbres de Lie paraboliques de $\mathfrak{s}\mathfrak{o}(n,\mathbb{K})$ qui sont  quasi-réductives (voir \cite{DKT}).\\
\\Tauvel et Yu   ont étudié dans \cite[ch.40]{Tauvel-Yu} une classe d'algèbres de Lie reliée à celle des algèbres de Lie quasi-réductives.
\begin{defi}Une forme linéaire $g \in \mathfrak{g}^{\ast}$ est dite stable s'il existe un voisinage $V$ de $g$ dans $\mathfrak{g}^{\ast}$ tel que, pour toute forme linéaire $f\in V$, les stabilisateurs $\mathfrak{g}(g)$ et $\mathfrak{g}(f)$ soient conjugués par le groupe adjoint algébrique de $\mathfrak{g}$.
\end{defi}
\begin{defi}
Une algèbre de Lie est dite stable si elle admet une forme linéaire stable.
\end{defi}
La notion de stabilité a été introduite par Kosmann et Sternberg dans \cite{Kosmann1974}. Il est clair qu'une forme linéaire stable est régulière. Ceci a été notre principale motivation pour étudier les formes linéaires stables. D'autre part, si $\mathfrak{g}$ est une algèbre de Lie quasi-réductive, elle admet des formes linéaires régulières de type réductif lesquelles sont stables (voir \cite{DKT}). Une algèbre de Lie quasi-réductive est donc stable. Par contre, il existe des algèbres de Lie stables qui ne sont pas quasi-réductives (voir exemple \ref{ex}).\\
\\Dans \cite{Tauvel-Yu2004}, Tauvel et Yu  ont donn\'{e} un exemple d'une sous-alg\`{e}bre parabolique de $\mathfrak{s}\mathfrak{o}(8,\mathbb{K})$ qui n'admet aucune forme lin\'{e}aire stable qui est en particulier non quasi-réductive. La partie $\text{(ii)}$ de la conjecture $5.6$ de \cite{Panyushev2005} revient à affirmer qu'une sous-alg\`{e}bre biparabolique d'une alg\`{e}bre de Lie r\'{e}ductive est stable si et seulement si elle est quasi-réductive. Le but principal de ce travail est de montrer cette assertion pour le cas des sous-alg\`{e}bres paraboliques d'une alg\`{e}bre de Lie orthogonale (voir th\'{e}or\`{e}me \ref{stab de p}). Pour d\'{e}montrer ceci, nous utilisons les r\'{e}sultats de \cite{DKT} concernant la classification des sous-alg\`{e}bres paraboliques de $\mathfrak{s}\mathfrak{o}(n,\mathbb{K})$ qui sont quasi-r\'{e}ductives.\\
\\Dans \cite{Dvorsky2003}, Dvorsky introduit certaines algèbres de Lie, notées  $\mathfrak{r}_{\mathcal{V}}$,  qui stabilisent une forme bilinéaire alternée de rang maximal et un drapeau en position générique. Dans \cite{DKT}, Duflo, Khalgui et Torasso constatent que la classification des sous-algèbres de Lie paraboliques des algèbres de Lie orthogonales se ramène à déterminer parmi ces sous-algèbres de Lie  $\mathfrak{r}_{\mathcal{V}}$ celles qui sont quasi-réductives. Dans ce travail, on démontre qu'une algèbre de Lie $\mathfrak{r}_{\mathcal{V}}$  est stable si et seulement si elle est quasi-réductive (voir théorème \ref{stab de r}). Pour démontrer ceci, nous utilisons les résultats de \cite{DKT} concernant la classification de ces algèbres de Lie qui sont quasi-réductives.\\
\\Pour une algèbre de Lie algébrique $\mathfrak{g}$, nous ramenons l'étude de la stabilité de $\mathfrak{g}$ au cas de rang nul (voir proposition \ref{prop1} et conséquence \ref{cons}). Ceci permet de ramener également l'étude de la stabilité au cas de rang nul pour les sous-algèbres paraboliques de $\mathfrak{s}\mathfrak{o}(n,\mathbb{K})$ (voir \ref{red.parab}) ainsi que pour les algèbres de Lie $\mathfrak{r}_{\mathcal{V}}$ (voir preuve du théorème \ref{stab de r}).
\\

\textbf{Remerciements.} Je remercie mes encadreurs Pierre Torasso et Mohamed-Salah Khalgui qui ont proposé ce travail. Je remercie également Michel Duflo pour d'utiles remarques. Ce travail a bénéficié du soutien des universités de Tunis El-Manar (bourse d'alternance) et Poitiers, du projet CMCU franco-tunisien PHC G1504 "Théorie de Lie, Systèmes intégrables, analyse stochastique", ainsi que du projet Erasmus Mundus Al-Idrisi.

\section{Stabilit\'{e} des algèbres de Lie alg\'{e}briques}\label{première section}
Dans toute la suite, $\mathbb{K}$ est un corps commutatif algébriquement clos de caractéristique nulle. Par algèbre de Lie algébrique, nous entendons une algèbre de Lie qui soit l'algèbre de Lie d'un groupe algébrique affine connexe défini sur $\mathbb{K}$.\\
Si $\mathbf{G}$ est un groupe algébrique affine, on note $^{u}\mathbf{G}$ le radical unipotent de $\mathbf{G}$, $\mathfrak{g}$ son algèbre de Lie, $^{u}\mathfrak{g}$ l'algèbre de Lie de $^{u}\mathbf{G}$. Rappelons que $\mathbf{G}$ admet une décomposition de \textit{Levi}: il existe un sous-groupe réductif $\mathbf{R} \subset \mathbf{G}$, appelé sous-groupe de \textit{Levi} de $\mathbf{G}$, dont la classe de conjugaison modulo $^{u}\mathbf{G}$ est uniquement déterminée, tel que $\mathbf{G}= \mathbf{R}\,^{u}\mathbf{G}$. Au niveau des algèbres de Lie, on a la décomposition de \textit{Levi} $\mathfrak{g}=\mathfrak{r}\,\oplus\,^{u}\mathfrak{g} $. On dit que $\mathfrak{r}$ est un facteur réductif de $\mathfrak{g}$ et que $^{u}\mathfrak{g}$ est son radical unipotent.
\subsection{Formes fortement régulières}\label{1.1}

Soient $\mathfrak{g}$ une algèbre de Lie, $\mathfrak{g}^{\ast}$ son dual. On fait opérer $\mathfrak{g}$ dans $\mathfrak{g}^{\ast}$ au moyen de la représentation coadjointe. Ainsi, si $X, Y \in \mathfrak{g}$ et $g \in \mathfrak{g}^{\ast}$, on a $$(X.g) (Y)=g([Y, X]).$$
Avec les notations précédentes, on définit une forme bilinéaire alternée $\Phi_{g}$ sur $\mathfrak{g}$ par:$$\Phi_{g}(X,Y)=g([X,Y]).$$
Le noyau de $\Phi_{g}$ est égal à $\mathfrak{g}(g)$. L'entier $$\ind(\mathfrak{g})=\rm{inf}\{\dim\,\mathfrak{g}(g);\,\,g \in \mathfrak{g}^{\ast} \}$$
est appelé l'indice de $\mathfrak{g}$.\\
L'indice de $\mathfrak{g}$, noté $\ind \mathfrak{g}$, est donc la dimension minimale des stabilisateurs dans $\mathfrak{g}$ d'un élément de $\mathfrak{g}^{\ast}$ pour l'action coadjointe. Il a été introduit par Dixmier dans \cite{Dixmier} pour son importance dans la théorie des représentations et la théorie des orbites. De plus, l'indice de $\mathfrak{g}$ est le degré de transcendance du corps des fractions rationnelles $G$-invariantes sur $\mathfrak{g}^{\ast}$.\\
\begin{defi} Une forme linéaire $g\in \mathfrak{g}^{\ast}$ est dite régulière si la dimension de son stabilisateur dans $\mathfrak{g}$ pour l'action coadjointe est égale à l'indice de $\mathfrak{g}$.
\end{defi}
\begin{rema}
Il est bien connu que l'ensemble $\mathfrak{g}^{\ast}_{reg}$ des éléments réguliers de $\mathfrak{g}^{\ast}$ est un ouvert de Zariski non vide de $\mathfrak{g}^{\ast}$.
\end{rema}
\begin{defi}\label{}
Soit $\mathfrak{g}$ une algèbre de Lie algébrique et $\mathbf{G}$ un groupe algébrique d'algèbre de Lie $\mathfrak{g}$. Une forme linéaire $g \in \mathfrak{g}^{\ast}$ est dite fortement régulière si elle est régulière, auquel cas $\mathfrak{g}(g)$
est une algèbre de Lie commutative (voir \cite{duflo-vergne-1969}), et si de plus le tore $\mathfrak{j}_{g}$, unique facteur réductif de $\mathfrak{g}(g)$, est de dimension maximale lorsque $g$ parcourt l'ensemble des formes  régulières.
\end{defi}
Cette définition est due à Duflo (voir \cite{duflo-ip-1982}).
\begin{rema}\label{}
Il est bien connu que l'ensemble des formes fortement régulières  est un ouvert de Zariski $\mathbf{G}$-invariant non vide de $\mathfrak{g}^{\ast}$.\\
Les tores $\mathfrak{j}_{g}$, $g \in \mathfrak{g}^{\ast}$ sont appelés les sous-algèbres de Cartan-Duflo de $ \mathfrak{g}$. Ils sont deux \`{a} deux conjugués sous l'action du groupe adjoint connexe de $\mathfrak{g}$.
\end{rema}
Les deux résultats de la remarque précédente sont énoncés sans démonstration dans \cite{duflo-ip-1982} (pour la démonstration voir  \cite{charbonnel-1982}).
\begin{defi}\label{1.3.1}
Soit $\mathfrak{g}$ une algèbre de Lie algébrique sur $\mathbb{K}$. On appelle rang de $\mathfrak{g}$ sur $\mathbb{K}$ et on note $\rang \mathfrak{g}$ la dimension commune de ses sous-algèbres de Cartan-Duflo.
\end{defi}
Soit $\mathfrak{g}$ une algèbre de Lie algébrique sur $\mathbb{K}$, $\mathfrak{n}$ un idéal de $\mathfrak{g}$ contenu dans $^{u}\mathfrak{g}$, $g \in \mathfrak{g}^{\ast}$ une forme linéaire fortement régulière, $n$ la restriction de $g$ à $\mathfrak{n}$, $\mathfrak{h}$ le stabilisateur de $n$ dans $\mathfrak{g}$ et $h$ la restriction de $g$ sur $\mathfrak{h}$.
\begin{lem}\label{duf82} \emph{(Voir  \cite[I.16]{duflo-1982})} On garde les notations précédentes. Alors, on a\\
$(\text{i})$ $exp(\mathfrak{n}(n)).g=g+(\mathfrak{h}+\mathfrak{n})^{\perp}$.\\
$(\text{ii})$ $\mathfrak{h}(h)=\mathfrak{g}(g)+\mathfrak{n}(n)$.
\end{lem}

\subsection{Algèbres de Lie quasi-réductives et  algèbres de Lie stables}\label{1.2}
\begin{defi}\label{}
Soit $\mathfrak{g}$ une algèbre de Lie algébrique.\\
$\text{(i)}$ Une forme linéaire $g \in \mathfrak{g}^{\ast}$ est dite de type réductif si le radical unipotent de $\mathfrak{g}(g)$ est central dans $\mathfrak{g}$.\\
$\text{(ii)}$ L'algèbre de Lie $\mathfrak{g}$ est dite quasi-réductive s'il existe une forme linéaire sur $\mathfrak{g}$ qui soit de type réductif.
\end{defi}
\begin{rema}\label{}
Soit $\mathfrak{g}$ une algèbre de Lie et $\mathfrak{z}$ son centre. Compte tenu de la définition précédente, il est clair qu'une forme linéaire $g \in \mathfrak{g}^{\ast}$ est de type réductif si et seulement si le radical unipotent de $\mathfrak{g}(g) $ qu'on note $^{u}(\mathfrak{g}(g))$ est égal au radical unipotent du centre $^{u}\mathfrak{z}$. Dans ce cas, $\mathfrak{g}(g)$ possède un unique facteur réductif, que l'on note $\mathfrak{r}_{g}$.\\
Soit $\mathbf{G}$ un groupe algébrique connexe d'algèbre de Lie $\mathfrak{g}$. Si une forme linéaire $g$ est de type réductif, il en est de m\^{e}me de toutes celles qui sont contenues dans son orbite sous l'action coadjointe de $\mathbf{G}$: l'orbite sous $\mathbf{G}$ d'une telle forme est dite de type réductif.
On voit alors qu'une forme linéaire $g \in \mathfrak{g}^{\ast}$ est dite de type réductif si et seulement si l'une des propriétés équivalentes suivantes est vérifiée:
\begin{eqnarray}
\begin{split}
&\mathbf{G}(g)/\mathbf{Z} \,\,\hbox {est  réductif },\\
&\mathfrak{g}(g)/\mathfrak{z}\,\,\, \hbox { est  réductive },\\
&^{u}(\mathbf{G}(g) )\subset \mathbf{Z},\\
&^{u}(\mathfrak{g}(g))\subset \mathfrak{z},\\
\end{split}
\end{eqnarray}
o\`{u} avec nos notations, $^{u}(\mathbf{G}(g))$ désigne le radical unipotent du stabilisateur $\mathbf{G}(g)$ et $\mathbf{Z}$ le centre de $\mathbf{G}$.
\end{rema}
\begin{rema}
Si $\mathfrak{g}$ est une algèbre de Lie quasi-réductive, les formes fortement régulières sont de type réductif (Voir\cite{DKT}), on voit donc que $\mathfrak{g}$ est quasi-réductive si et seulement si $\ind\mathfrak{g}=\rang\mathfrak{g}+\dim\,^{u}\mathfrak{z}$
\end{rema}
\begin{ex}\label{}
Voici quelques exemples d'alg\`{e}bres de Lie quasi-r\'{e}ductives.\\
\begin{itemize}
\item
Toute alg\`{e}bre de Lie r\'{e}ductive est quasi-r\'{e}ductive puisque la forme lin\'{e}aire $0 \in \mathfrak{g}^{\ast}$ est de type r\'{e}ductif.
\item
Toutes les sous-alg\`{e}bres de Borel d'une alg\`{e}bre de Lie simple sont  quasi-r\'{e}ductives ( Kostant \cite{kostant2011cascade}, non publi\'{e}, voir  \cite{joseph1977}).
\item
Les sous-alg\`{e}bres paraboliques d'une alg\`{e}bre de Lie simple $\mathfrak{s}$ de type $\mathrm{A}$ ou $\mathrm{C}$ sont quasi-r\'{e}ductives (Panyushev \cite{Panyushev2005}).
\end{itemize}
\end{ex}
\begin{rema} Pour le cas des sous-algèbres paraboliques de $\mathfrak{s}\mathfrak{o}(n, \mathbb{K})$, $n\geq 7$, Panyushev dans \cite{Panyushev2005} et Dvorsky dans \cite{Dvorsky2003} pr\'{e}sentent de nombreux exemples de sous-alg\`{e}bres paraboliques de $\mathfrak{s}\mathfrak{o}(n,\mathbb{K})$ qui sont quasi-r\'{e}ductives. Dans \cite{DKT}, on trouve une caract\'{e}risation des sous-alg\`{e}bres paraboliques de $\mathfrak{s}\mathfrak{o}(n,\mathbb{K})$ qui sont quasi-r\'{e}ductives. De plus, les sous-alg\`{e}bres paraboliques d'une alg\`{e}bre de Lie simple exceptionnelle $\mathfrak{s}$ ne sont pas toujours quasi-r\'{e}ductives. Dans \cite{Anne+Karin},  Baur et Moreau ont termin\'{e} la classification des sous-alg\`{e}bres paraboliques quasi-r\'{e}ductives dans le cas des alg\`{e}bres de Lie simples exceptionnelles.
\end{rema}
\begin{defi}\label{}
Soit $\mathfrak{g}$ une algèbre de Lie.\\
$\text{(i)}$ On dit que $f \in \mathfrak{g}^{\ast}$ est stable s'il existe un voisinage $W$ de $f$ dans $\mathfrak{g}^{\ast}$ tel que, pour tout $g \in W$, les stabilisateurs $\mathfrak{g}(f)$ et $\mathfrak{g}(g)$ soient $\mathbf{K}$-conjugués, o\`{u}  $\mathbf{K}$ désigne le groupe adjoint algébrique de $\mathfrak{g}$.\\
$\text{(ii)}$ Une algèbre de Lie $\mathfrak{g}$ est dite "stable" si elle admet une forme linéaire stable.
\end{defi}
\begin{lem}\label{Tauvel Yu}\emph{(Voir \cite{Tauvel-Yu2004})}
Soit $f$ un élément de $\mathfrak{g}^{\ast}$.\\
Si $\mathfrak{g}$ est une algèbre de Lie algébrique, f est stable si et seulement si $[\mathfrak{g},\mathfrak{g}(f)] \cap \mathfrak{g}(f)=\{0\}$.
\end{lem}

\begin{rema}\label{}
Si $\mathfrak{g}$ est une algèbre de Lie quasi-réductive, les formes linéaires régulières de type réductif, qui sont exactement les formes fortement régulières, sont stables. Par suite, toute algèbre de Lie quasi-réductive est stable. Cependant, il existe des algèbres de Lie stables  qui ne sont pas quasi-réductives.
\end{rema}

\subsubsection{Exemple}\label{ex}
Voici un exemple d'algèbre de Lie stable qui n'admet aucune forme linéaire de type réductif.\\
Soit $\mathfrak{g}$ l'algèbre de Lie produit semi-direct de $\mathfrak{s}\mathfrak{o}(2)$ par $\mathbb{K}^{2}$ de sorte que $\mathfrak{g}=\mathfrak{s}\mathfrak{o}(2)\oplus\mathbb{K}^{2}$, le crochet étant donné par $$[t\,\, W+ v, z \,\,W+v^{\prime}]=t \,\,W.v^{\prime}- z\,\, W. v,\,\,\,t,\,z \in \mathbb{K},\,\, v,\,v^{\prime} \in \mathbb{K}^{2},$$
o\`{u} $\mathfrak{s}\mathfrak{o}(2)=\mathbb{K} W$, avec $W=\left(
                                                                                                                            \begin{array}{cc}
                                                                                                                              0 & 1 \\
                                                                                                                              -1 & 0 \\
                                                                                                                            \end{array}
                                                                                                                          \right)
$.
On vérifie aisément que le centre de $\mathfrak{g}$ est trivial. Par suite, $\mathfrak{g}$ est quasi-réductive si et seulement si $\ind\mathfrak{g}= \rang \mathfrak{g}$.\\
Soit $g \in \mathfrak{g}^{\ast}$, $n$ la restriction de $g$ au  radical unipotent $^{u}\mathfrak{g}=\mathbb{K}^{2}$ de $\mathfrak{g}$. On suppose que $n\neq 0$. Alors
 $$X= t\, W+ v\in \mathfrak{g}(g) \hbox{ si et seulement si pour tout }\, z \in\mathbb{K}\,\,et\,\,v^{\prime} \in\mathbb{K}^{2} \,\,\langle n,t \,W.v^{\prime}- z\, W. v\rangle=0.$$
Faisant  $z=0$, on voit que si $X \in  \mathfrak{g}(g)$, alors  pour tout $v^{\prime} \in \mathbb{K}^{2}$,  on a $$t \,\,\langle n, \,\,W.v^{\prime}\rangle=0.$$
Il suit alors de ce qui précède que $$X= t\,\, W+ v \in \mathfrak{g}(g)  \hbox{ si et seulement si } t=0 \hbox { et } \langle n, z\,\,W.v\rangle=0 \hbox{ pour tout } z\in \mathbb{K}.$$
Par suite, $\mathfrak{g}(g)=(W.n)^{\bot} \subset \ ^{u}\mathfrak{g}$. Ceci montre  que l'algèbre de Lie $\mathfrak{g}$ est non quasi-réductive.\\
D'autre part, soit $O=\{n \in (^{u}\mathfrak{g})^{\ast}/ \langle n, e_{1}\rangle^{2}+\langle n, e_{2} \rangle^{2} \neq 0\}$, avec $(e_{1}, e_{2})$ la base canonique de $\mathbb{K}^{2}$. On vérifie aisément que $O$ est un ouvert de Zariski non vide. Si $n,\,n^{\prime}\in O$, il existe $k \in \mathbf{S}\mathbf{O}(2)$ tel que $k. \mathbb{K}n=\mathbb{K}n^{\prime}$. Ceci montre que $\mathfrak{g}$ est stable.

\subsection{Réduction au cas de rang nul}\label{}
Nous allons maintenant énoncer le résultat suivant qui permet de ramèner l'étude de la stabilité des algèbres de Lie algébriques aux algèbres de Lie algébriques de \rang nul et qui nous sera utile pour la suite.\\
\\Soit $ \mathfrak{j}$ une sous-algèbre de Cartan-Duflo de $\mathfrak{g}$. On a $\mathfrak{g}=\mathfrak{g}^{\mathfrak{j}}\oplus[\mathfrak{j},\mathfrak{g}]$, o\`{u} $\mathfrak{g}^{\mathfrak{j}}$ désigne le centralisateur de $\mathfrak{j}$ dans $\mathfrak{g}$. On identifie le dual de l'algèbre de Lie $\mathfrak{g}^{\mathfrak{j}}$ à l'orthogonal de $[\mathfrak{j}, \mathfrak{g}]$ dans $\mathfrak{g}^{\ast}$ qui n'est autre que $\mathfrak{g}^{\ast \mathfrak{j}}$, l'ensemble des point fixes de $\mathfrak{j}$ dans $\mathfrak{g}^{\ast}$.

\begin{pr}\label{prop1}
Soit $\mathfrak{g}$ une algèbre de Lie algébrique, $\mathfrak{j}$ une sous-algèbre de Cartan-Duflo de $\mathfrak{g}$ et   $\mathfrak{g}^{\mathfrak{j}}$ le centralisateur de $\mathfrak{j}$ dans $\mathfrak{g}$. Alors
$$\mathfrak{g} \hbox{ est stable si et seulement si } \mathfrak{g}^{\mathfrak{j}} \hbox{ est stable  }.$$
\end{pr}

\begin{proof}
Soit  $\mathfrak{g}$ une algèbre de Lie algébrique admettant des formes stables. Alors l'ensemble $\mathfrak{g}^{\ast}_{s}$ des formes stables est un ouvert de Zariski non vide. Par suite, $\mathfrak{g}^{\ast}_{s}\cap \mathfrak{g}^{\ast}_{r}$ est un ouvert de Zariski non vide, o\`{u} $\mathfrak{g}^{\ast}_{r}$ est l'ouvert de Zariski des formes fortement régulières.\\
Soit $g \in \mathfrak{g}^{\ast}_{s}\cap \mathfrak{g}^{\ast}_{r}$ et $\mathfrak{j} \subset \mathfrak{g}(g)$ le facteur réductif de $\mathfrak{g}(g)$. Alors $g \in \mathfrak{g}^{\ast \mathfrak{j}}$ et $g$ est une forme stable dans $\mathfrak{g}^{\ast \mathfrak{j}}$. En effet:\\
$$\mathfrak{g}^{ \mathfrak{j}}(g)=\mathfrak{g}(g) \,\,\,et\,\,\, [\mathfrak{g}^{ \mathfrak{j}},\mathfrak{g}(g)]\subset [\mathfrak{g},\mathfrak{g}(g)].$$
Réciproquement, soit $\mathfrak{j} \subset \mathfrak{g}$ une sous-algèbre de Cartan-Duflo de $\mathfrak{g}$ et supposons que $\mathfrak{g}^{\mathfrak{j}}$ admette des formes stables. Alors $(\mathfrak{g}^{\ast \mathfrak{j}})_{s}\cap \mathfrak{g}^{\ast}_{r}$ est un ouvert de Zariski non vide de $\mathfrak{g}^{\ast\mathfrak{j}}$. Soit $g \in (\mathfrak{g}^{\ast\mathfrak{j}})_{s}\cap \mathfrak{g}^{\ast}_{r}$. Alors on a  $$\mathfrak{g}(g)\cap [\mathfrak{g}, \mathfrak{g}(g)]=\mathfrak{g}(g)\cap [\mathfrak{g}, \mathfrak{g}(g)]^{\mathfrak{j}}=\mathfrak{g}(g)\cap [\mathfrak{g}^{\mathfrak{j}}, \mathfrak{g}(g)]=\{0\},$$
de sorte que $g$ est stable pour $\mathfrak{g}$. D'o\`{u} la proposition
\end{proof}

\begin{cons}\label{cons}
Soit maintenant $\mathfrak{g}^{\mathfrak{j}, 1}\subset \mathfrak{g}^{\mathfrak{j}}$ l'idéal orthogonal de $\mathfrak{j}$ pour la forme $(X,Y)\mapsto tr(XY)$. Alors $\mathfrak{g}^{\mathfrak{j}}=\mathfrak{j}\oplus \mathfrak{g}^{\mathfrak{j},1}$ et il est clair que $g \in \mathfrak{g}^{\ast \mathfrak{j}}$ est une forme stable si et seulement si $g^{\prime}=g|_{\mathfrak{g}^{\mathfrak{j},1}}$ est une forme stable pour $\mathfrak{g}^{\mathfrak{j},1}$.\\
Il suit  de ce qui précède que $$\mathfrak{g} \hbox{ est stable si et seulement si } \mathfrak{g}^{\mathfrak{j},1} \hbox{ est stable }$$
Comme $\mathfrak{g}^{\mathfrak{j},1}$ est de {\em rang} nul , on se ramène donc  \`{a} étudier la stabilité pour les algèbres de Lie de {\em rang} nul.
\end{cons}

\section{Cas des algèbres de Lie $\mathfrak{r}_{\mathcal{V}}$}\label{première section}
Dans \cite{DKT}, Duflo, Khalgui et Torasso constatent que la classification des sous-algèbres paraboliques quasi-réductives des algèbres de Lie simples classiques se ramène à determiner parmi les algèbres de Lie qui stabilisent une forme bilinéaire alternée de rang maximal et un drapeau en position générique, celles qui sont quasi-réductives. Le but de ce numéro est de caractériser parmi ces sous-algèbres celles qui sont stables.\\
\subsection{}\label{}
Soit $V$ un espace vectoriel de dimension finie sur $\mathbb{K}$ et soit $\mathcal{V}=\{ \{0\}=V_{0} \varsubsetneq V_{1} \varsubsetneq V_{2}\varsubsetneq...\varsubsetneq V_{t-1} \varsubsetneq V_{t}=V \}$ un drapeau de $V$. On désigne par $\mathfrak{q}_{\mathcal{V}}$ la sous-algèbre parabolique de $\mathfrak{g}\mathfrak{l}(V)$ constituée des endomorphismes qui laissent invariant $\mathcal{V}$ et par $\mathbf{Q}_{\mathcal{V}}$ le sous-groupe algébrique connexe de $\mathbf{G}\mathbf{L}(V)$ d'algèbre de Lie $\mathfrak{q}_{\mathcal{V}}$.
\begin{defi}\label{1.2}
Soient $V$ un espace vectoriel non nul de dimension finie sur $\mathbb{K}$, \,\,\,$\mathcal{V}=\{ \{0\}=V_{0} \varsubsetneq V_{1}\varsubsetneq V_{2}\varsubsetneq...\varsubsetneq V_{t-1} \varsubsetneq V_{t}=V\}$ un drapeau de $V$ et $\xi$ une forme bilinéaire alternée sur $V$. On dit que le drapeau $\mathcal{V}$ est générique relativement \`{a} $\xi$ ou que la forme $\xi$ est générique relativement \`{a} $\mathcal{V}$, si pour $1\leq i \leq t$, la restriction de $\xi$ \`{a} $V_{i}$ est de rang maximum, savoir $2 [\frac{1}{2}\dim V_{i}]$, et $V_{i}^{\bot_{\xi}} \cap V_{i-1}=\{0\}$ o\`{u} $V_{i}^{\bot_{\xi}}$ désigne l'orthogonal de $V_{i}$ dans $V$ relativement à  $\xi$.
\end{defi}
\begin{defi}\label{}
Soient $V$ un espace vectoriel non nul de dimension finie  sur $\mathbb{K}$, $b=(e_{1},...,e_{r})$ une base de $V$ et $\mathcal{V}=\{ \{0\}=V_{0} \varsubsetneq V_{1}\varsubsetneq V_{2}\varsubsetneq...\varsubsetneq V_{t-1} \varsubsetneq V_{t}=V\}$ un drapeau de $V$. On dit que la base $b$ est adaptée au drapeau $\mathcal{V}$, si pour tout $1\leq i\leq t$ l'espace  $V_{i}$ est engendré par la famille $(e_{1},...,e_{\dim V_{i}})$.
\end{defi}
\begin{lem}\label{1.2}\emph{(Voir \cite[Lemme 4.2.1]{DKT})}
Soit $V$ un espace vectoriel de dimension finie sur $\mathbb{K}$ et $\mathcal{V}=\{ \{0\}=V_{0} \varsubsetneq V_{1} \varsubsetneq V_{2}\varsubsetneq...\varsubsetneq V_{t-1} \varsubsetneq V_{t}=V\}$ un drapeau de $V$. Alors\\
$\text{(i)}$ une forme $\xi \in \bigwedge^{2}V^{\ast}$ est générique relativement \`{a} $\mathcal{V}$, si et seulement s'il existe une base $e_{1},...,e_{r}$ de $V$ adaptée au drapeau $\mathcal{V}$ telle que, notant $e_{1}^{\ast},...,e_{r}^{\ast}$ la base duale, on ait $\xi= \sum_{1\leq2i \leq r}e_{2i-1}^{\ast} \wedge e_{2i}^{\ast}$.\\
$\text{(ii)}$ l'ensemble des formes bilinéaires alternées $\xi$ sur $V$ qui sont génériques relativement \`{a} $\mathcal{V}$ est une orbite ouverte sous l'action de $\mathbf{Q}_{\mathcal{V}}$ dans $\bigwedge^{2}V^{\ast}$.
\end{lem}
\begin{defi}\label{1.2}
Soit $\mathcal{V}=\{ \{0\}=V_{0} \varsubsetneq V_{1}\varsubsetneq V_{2}\varsubsetneq...\varsubsetneq V_{t-1} \varsubsetneq V_{t}=V \}$ un drapeau. On note $h=h(\mathcal{V})$ le nombre d'indices $i$ tels que $1 \leq i \leq t-1$ et $\dim V_{i}$ et $\dim V_{i+1}$ soient tous deux impairs. On dit que le drapeau $\mathcal{V}$ vérifie la propriété $\mathcal{P}$, si deux espaces consécutifs de la suite $\mathcal{V}$ ne peuvent \^{e}tre tous deux de dimension impaire, c'est \`{a} dire $h(\mathcal{V})=0$.
\end{defi}
\begin{defi}\label{}
Soit $\mathcal{V}=\{ \{0\}=V_{0} \varsubsetneq V_{1}  \varsubsetneq V_{2}\varsubsetneq...\varsubsetneq V_{t-1} \varsubsetneq V_{t}=V\}$ un drapeau de $V$. On dit que le drapeau $\mathcal{V}$ vérifie la condition $(\ast)$, si, pour $1\leq i \leq t-1$, entre deux sous-espaces consécutifs $V_{i}$ et $V_{i+1}$ le saut de dimension est 2 s'ils sont tous deux de dimension impaire et 1 sinon.
\end{defi}

\subsection{}\label{}
Soit  $V$ un espace vectoriel de dimension finie sur $\mathbb{K}$  muni d'une forme bilinéaire alternée $\xi$ de rang
maximal. On désigne par $\mathfrak{g}\mathfrak{l}(V) (\xi)$ l'annulateur de $\xi$ dans $\mathfrak{g}\mathfrak{l}(V)$. Lorsque $\xi$ est symplectique, $\mathfrak{g}\mathfrak{l}(V) (\xi)$ est noté $\mathfrak{s}\mathfrak{p}( V, \xi)$ ou simplement $\mathfrak{s}\mathfrak{p}( V )$ et c'est l'algèbre de Lie
du groupe symplectique correspondant. On se donne  un drapeau $\mathcal{V}$ de $V$ générique relativement à $\xi$  et on désigne par $\mathfrak{r}_{\mathcal{V}}$
la sous-algèbre de Lie de $\mathfrak{g}\mathfrak{l}(V) (\xi)$ constituée des endomorphismes stabilisant $\mathcal{V}$ qui est aussi la sous-algèbre de Lie $\mathfrak{q}_{\mathcal{V}}(\xi)$  de  $\mathfrak{q}_{\mathcal{V}}$ stabilisant $\xi$ .

\begin{pr}\label{r de v}\emph{(Voir\cite[Proposition 5.3.1]{DKT})} Soit  $V$  un espace vectoriel de dimension finie  sur $\mathbb{K}$  muni d'une forme bilinéaire alternée $\xi$ de rang
maximal  et $\mathcal{V}=\{ \{0\}=V_{0} \varsubsetneq V_{1}  \varsubsetneq V_{2}\varsubsetneq...\varsubsetneq V_{t-1} \varsubsetneq V_{t}=V\}$ un drapeau générique relativement \`{a} $\xi$.\\
$\text{a})$ Supposons que $\dim V$ est pair, de sorte que $\xi$ est symplectique. Alors,

$\text{(i)}$ on a $\ind(\mathfrak{r}_{\mathcal{V}}) =\sum_{i=1}^{t} [\frac{1}{2} (\dim V_{i}-\dim V_{i-1})]$,

$\text{(ii)}$ la dimension du radical unipotent d'un stabilisateur générique

de la représentation coadjointe de $\mathfrak{r}_{\mathcal{V}}$ est $h(\mathcal{V})$,

$\text{(iii)}$ l'algèbre de Lie $\mathfrak{r}_{\mathcal{V}}$ est quasi-réductive si et seulement si le drapeau $\mathcal{V}$

 vérifie la propriété $\mathcal{P}$.\\
$\text{b})$ Supposons que $\dim V$ est impair et  désignons par $\mathcal{V}^{\prime}$  le drapeau du sous-espace $V_{t-1}$ obtenu en supprimant l'espace $V$ du drapeau $\mathcal{V}$. Alors,

$\text{(iv)}$ on a $$\ind(\mathfrak{r}_{\mathcal{V}}) =\left\{
                                                \begin{array}{ll}
                                                  \sum_{i=1}^{t} [\frac{1}{2} (\dim V_{i}-\dim V_{i-1})]-1 & \hbox{ si $\dim V_{t-1}< \dim V-1$,}\\
                                                  \sum_{i=1}^{t} [\frac{1}{2} (\dim V_{i}-\dim V_{i-1})]+1 & \hbox{ si $\dim V_{t-1}= \dim V-1$,}
                                                \end{array}
                                              \right.$$

$\text{(v)}$ La dimension du radical unipotent d'un stabilisateur générique de la

 représentation coadjointe de $\mathfrak{r}_{\mathcal{V}}$ est $h(\mathcal{V}^{\prime})$.

$\text{(vi)}$ L'algèbre de lie $\mathfrak{r}_{\mathcal{V}}$ est quasi-réductive si et seulement si le drapeau $\mathcal{V}^{\prime}$

vérifie la propriété $\mathcal{P}$.
\end{pr}

\begin{rema}\label{2.1}
Soit  $V$  un espace vectoriel de dimension finie sur $\mathbb{K}$  muni d'une forme bilinéaire alternée  $\xi$ de rang
maximal et $\mathcal{V}=\{ \{0\}=V_{0} \varsubsetneq V_{1}  \varsubsetneq V_{2}\varsubsetneq...\varsubsetneq V_{t-1} \varsubsetneq V_{t}=V\}$ un drapeau générique relativement \`{a} $\xi$.\\
Supposons qu'il existe un sous-espace $V_{i}$ du drapeau $\mathcal{V}$ tel que $\dim V_{i}$ soit pair. Alors $\xi$ induit une forme symplectique sur le sous-espace $U=  V_{i}$ de sorte que $V$ est somme directe de $U$ et de son orthogonal $W$ relativement \`{a} $\xi$. On désigne par $\mathcal{U}$ (resp. $\mathcal{W}$)  le drapeau $\mathcal{U}=\{ \{0\}=V_{0} \varsubsetneq V_{1}  \varsubsetneq V_{2}\varsubsetneq...\varsubsetneq V_{i-1} \varsubsetneq V_{i}=U\}$ ( resp. $\mathcal{W}=\{ \{0\}=W_{0} \varsubsetneq W_{1}= V_{i+1}\cap W  \varsubsetneq...\varsubsetneq V_{j}=V_{i+j}\cap W \varsubsetneq... \varsubsetneq W_{t-i}=W\}$ ). Comme $\mathcal{V}$ est générique relativement \`{a} $\xi$, il est clair que $\mathcal{U}$ ( resp. $\mathcal{W}$ ) est générique relativement \`{a} $\xi_{|U}$ ( resp.$\xi_{|W}$ ). Par suite, la sous-algèbre  $\mathfrak{r}_{\mathcal{V}}$ s'identifie au produit direct $\mathfrak{r}_{\mathcal{U}} \times \mathfrak{r}_{\mathcal{W}}$. En particulier,  $\mathfrak{r}_{\mathcal{V}}$ est quasi-réductive (resp.stable) si et seulement si $\mathfrak{r}_{\mathcal{U}}$ et $\mathfrak{r}_{\mathcal{W}}$ le sont. De plus, on a $$\ind\mathfrak{r}_{\mathcal{V}}=\ind\mathfrak{r}_{\mathcal{U}}+\ind\mathfrak{r}_{\mathcal{W}}\,\, \hbox{ et }\,\, \rang\mathfrak{r}_{\mathcal{V}}=\rang\mathfrak{r}_{\mathcal{U}}+\rang\mathfrak{r}_{\mathcal{W}}.$$
D'autre part, $\mathcal{V}$ vérifie la condition $(\ast)$ si et seulement s'il en est de même de $\mathcal{U}$ et de $\mathcal{W}$.
\end{rema}
\begin{pr}\label{quasi}
Soit  $V$  un espace vectoriel de dimension finie sur $\mathbb{K}$  muni d'une forme bilinéaire alternée $\xi$ de rang
maximal et $\mathcal{V}=\{ \{0\}=V_{0} \varsubsetneq V_{1}  \varsubsetneq V_{2}\varsubsetneq...\varsubsetneq V_{t-1} \varsubsetneq V_{t}=V\}$ un drapeau de $V$ générique relativement \`{a} $\xi$.\\
 Alors les assertions suivantes sont équivalentes:\\
$\text{(i)}$ $\mathfrak{r}_{\mathcal{V}}$  est de $\rang$ nul.\\
$\text{(ii)}$ ${\mathcal{V}}$ vérifie la condition $(\ast)$ et si $\dim V$ est impair, il en est de m\^{e}me de $\dim V_{t-1}$.
\end{pr}
\begin{proof}
D'après la remarque  précédente, on se ramène à démontrer la proposition dans l'un des cas suivants:
\begin{eqnarray}
\begin{split}
&- \dim V  \hbox{ est pair } , t\geq 2 \hbox{  et pour } 1\leq i \leq t-1, \dim V_{i} \hbox{ est un nombre impair },\\
&-  \dim V_{i} \hbox{ impair pour }  1\leq i \leq t.\\
\end{split}
\end{eqnarray}
On pose $\dim V_{i}=2p_{i}+1$, $1\leq i \leq t-1$ et  $a_{i}=\dim V_{i}-\dim V_{i-1}$, $1\leq i \leq t$. Alors $a_{1}=2p_{1}+1=2q_{1}+1$ est impair et
$a_{i}=2(p_{i}-p_{i-1})=2(q_{i}+1)$, $2\leq i \leq t-1$ est pair.\\
On se place dans le premier cas et on pose $\dim V_{t}=2p_{t}$. Alors $a_{t}=2(p_{t}-p_{t-1})-1=2q_{t}+1$ est impair. D'après la proposition \ref{r de v},  $$\rang \mathfrak{r}_{\mathcal{V}}=\sum_{i=1}^{t}q_{i}.$$
Ainsi, $\mathfrak{r}_{\mathcal{V}}$ est de \rang nul si et seulement si $q_{i}=0$ pour tout $1\leq i\leq t$. D'o\`{u} la proposition dans ce cas.\\
On se place dans la deuxième cas. On pose $\dim V=2p_{t}+1$. Alors $a_{t}=2(p_{t}-p_{t-1})=2(q_{t}+1)$ est pair. D'après la proposition \ref{r de v}, on a $$\rang \mathfrak{r}_{\mathcal{V}}=\sum_{i=1}^{t}q_{i}.$$
La proposition est alors claire dans ce cas.
\end{proof}
\begin{theo}\label{stab de r}
Soit  $V$  un espace vectoriel de dimension finie sur $\mathbb{K}$  muni d'une forme bilinéaire alternée $\xi$ de rang
maximal  et $\mathcal{V}=\{ \{0\}=V_{0} \varsubsetneq V_{1}  \varsubsetneq V_{2}\varsubsetneq...\varsubsetneq V_{t-1} \varsubsetneq V_{t}=V\}$ un drapeau générique relativement \`{a} $\xi$.\\
Alors  les assertions suivantes sont équivalentes:\\
$\text{(i)}$ $\mathfrak{r}_{\mathcal{V}}$ admet une forme linéaire stable.\\
$\text{(ii)}$ $\mathfrak{r}_{\mathcal{V}}$ est quasi-réductive.
\end{theo}
\begin{proof}
Seule l'implication $\text{(i)} \Rightarrow \text{(ii)}$ mérite démonstration. D'après la remarque précédente, on se ramène au cas o\`{u} aucun des éléments du drapeau $\mathcal{V}$, \`{a} l'exception de $\{0\}$ et $V$, est de dimension paire. On pose $d=\dim V$ et $\dim V_{i}=2p_{i}+1$, $1\leq i \leq t$, $p_{t}=[\frac{d}2{}]$. On choisit une base $e_{1},...,e_{d}$ de $V$ adaptée au drapeau $\mathcal{V}$ et telle que $\xi=\sum_{1\leq i \leq p_{t}} e^{\ast}_{2i-1} \wedge e^{\ast}_{2i}$ et on pose $W_{i}=\mathbb{K} e_{2p_{i-1}+3}\oplus\mathbb{K} e_{2p_{i-1}+4}\oplus...\oplus\mathbb{K} e_{2p_{i}}$, $1\leq i \leq t$ (on convient que $p_{0}=-1$). Alors les $W_{i}$, $1\leq i \leq t$, sont des sous-espaces symplectiques deux à deux orthogonaux et il suit de \cite[lemme 26 et paragraphes 5.4, 5.5]{DKT} que les sous-algèbres de Cartan de $\prod_{1\leq i \leq t}\mathfrak{s}\mathfrak{p}(W_{i})$ sont des sous-algèbres de Cartan-Duflo de $\mathfrak{r}_{\mathcal{V}}$.\\
Soit donc $\mathfrak{j}$ une sous-algèbre de Cartan de $\prod_{1\leq i \leq t}\mathfrak{s}\mathfrak{p}(W_{i})$. Le sous-espaces $V^{0}$ de $V$ constitué des vecteurs $\mathfrak{j}$-invariants n'est autre que $(\oplus_{1\leq i \leq t}W_{i})^{\perp}$. On considère $\mathfrak{g}\mathfrak{l}(V^{0})$  comme une  sous-algèbre de $\mathfrak{g}\mathfrak{l}(V)$ en identifiant un élément $X \in \mathfrak{g}\mathfrak{l}(V^{0})$ avec l'endomorphisme linéaire de $V$ agissant comme $X$ dans $V^{0}$ et trivialement dans son supplémentaire $\oplus_{1\leq i \leq t} W_{i}$. Alors, on vérifie que le commutant de $\mathfrak{j}$ dans $\mathfrak{g}\mathfrak{l}(V)$ est $\mathfrak{j}\oplus \mathfrak{g}\mathfrak{l}(V^{0})$. Il est alors clair que $\mathfrak{r}_{\mathcal{V}}^{\mathfrak{j}}=\mathfrak{j}\oplus \mathfrak{r}_{\mathcal{V}^{0}}$, o\`{u} $\mathcal{V}^{0}$ est la trace du drapeau $\mathcal{V}$ sur le sous-espace $V^{0}$.\\
D'autre part, remarquons que $\xi_{0}=\xi|_{V^{0}}$ est de rang maximal et $\mathcal{V}^{0}$ est en position générique par rapport à $\xi_{0}$ (i.e $\mathfrak{r}_{\mathcal{V}^{0}}$ est de même type que $\mathfrak{r}_{\mathcal{V}}$). Compte tenu de la proposition \ref{prop1} et comme $\mathfrak{r}_{\mathcal{V}^{0}}$ est de \rang nul, on se ramène \`{a} traiter le cas de \rang nul pour $\mathfrak{r}_{\mathcal{V}}$.\\
Il s'agit donc de montrer que si $\rang \mathfrak{r}_{\mathcal{V}}$ est nul  et $\ind \mathfrak{r}_{\mathcal{V}}$ est strictement positif, alors $\mathfrak{r}_{\mathcal{V}}$ n'admet pas de forme stable.\\
Compte tenu de ce qui précède, il suffit  d'étudier les deux cas suivants:\\
\\(a) $\dim V=2p$, $p\geq 2$ et $\dim V_{i}=2i-1$ pour tout $1\leq i \leq t-1$ de sorte que  $\mathcal{V}=\{ \{0\}=V_{0} \varsubsetneq V_{1} \varsubsetneq V_{2}\varsubsetneq...\varsubsetneq V_{p}\varsubsetneq V_{p+1}=V \}$. On choisit une base $(e_{1},...,e_{2p})$ de $V$ adaptée au drapeau $\mathcal{V}$ telle que $V_{j}=\mathbb{K}e_{1}\oplus...\oplus \mathbb{K}e_{2j-1},  \,\,\,1\leq j \leq p$ et $\xi=\sum_{j=1}^{p} e^{\ast}_{2j-1} \wedge e^{\ast}_{2j}$. On note $(E_{ij})_{1\leq i,j \leq 2p}$ la base de $\mathfrak{g}\mathfrak{l}(V)$ canoniquement associée à la base $(e_{1},...,e_{2p})$ de $V$.\\
Dans ces conditions, il résulte de  \cite[5.4]{DKT} que
\begin{equation}
\mathfrak{g}=\mathfrak{r}_{\mathcal{V}}=\mathfrak{l}_{\mathcal{V}} \oplus \mathfrak{n}_{\mathcal{V}}
\end{equation}
o\`{u}
\begin{equation}
\mathfrak{l}_{\mathcal{V}}= \bigoplus_{j=1}^{p} \mathbb{K} H_{j},
\end{equation}
avec
$$ H_{j}=E_{2j-1,2j-1}-E_{2j,2j},$$
est un facteur réductif de $\mathfrak{r}_{\mathcal{V}}$, commutatif dans ce cas, et
\begin{equation}
\mathfrak{n}_{\mathcal{V}}= \bigoplus_{j=1}^{p} \mathbb{K} Z_{j}\oplus (\bigoplus_{j=1}^{p-1}\mathbb{K} T_{j}),
\end{equation}
avec
 $$ Z_{j}=2 E_{2j-1,2j} \hbox{ et } T_{j}=E_{2j-1,2j+2}+E_{2j+1,2j},$$
est le radical unipotent de $\mathfrak{r}_{\mathcal{V}}$, également commutatif dans ce cas.\\
On a alors:
\begin{eqnarray}
\begin{split}
[H_{j}, Z_{l}]&=2 \delta_{l,j} Z_{j}\\
[H_{j}, T_{l}]&= \delta_{l,j} T_{j}+ \delta_{l,j-1} T_{j-1}=(\delta_{j,l}+ \delta_{j,l+1}) T_{l}.\\
\end{split}
\end{eqnarray}
Soit $g \in \mathfrak{g}^{\ast}$. On pose $n =g|_{n_{\mathcal{V}}}$, $g(Z_{j})=\zeta_{j}$ et $g(T_{j})= \tau_{j}$.\\
Supposons que  $\prod_{j=1}^{p} \zeta_{j}\neq 0$. Alors, on a $$\mathfrak{g}(n)=\mathfrak{n}_{\mathcal{V}}\,\,et\,\,\mathfrak{g}(g)=\mathfrak{n}_{\mathcal{V}}(g).$$
Par suite et compte tenu du lemme \ref{duf82}, on a $\mathbf{N}_{\mathcal{V}}.g=g + \mathfrak{n}_{\mathcal{V}}^{\perp}$, o\`{u} $\mathbf{N}_{\mathcal{V}}$ désigne le radical unipotent du sous-groupe algébrique connexe $\mathbf{R}_{\mathcal{V}}$ de $\mathbf{G}\mathbf{L}(V)$ d'algèbre de Lie $\mathfrak{r}_{\mathcal{V}}$. Ainsi, on peut supposer que $g|_{\mathfrak{l}_{\mathcal{V}}}=0$.\\
Soit  $$X=\sum_{i=1}^{p} z_{i}\,\,Z_{i}+ \sum_{i=1}^{p-1} t_{i}\,\,T_{i} \in \mathfrak{n}_{\mathcal{V}}.$$
On a
\begin{eqnarray}
\begin{split}
[H_{1}, X]  &=2 z_{1} Z_{1}+ t_{1} T_{1}\\
[H_{i}, X] &=2 z_{i} Z_{i}+ t_{i} T_{i}+  t_{i-1} T_{i-1},\,\,2\leq i \leq p-1,\\
[H_{p}, X] &=2 z_{p} Z_{p}+ t_{p-1} T_{p-1}
\end{split}
\end{eqnarray}
Alors, on a $X \in \mathfrak{n}_{\mathcal{V}}(g)$ si et seulement si:
$$\left\{
                           \begin{array}{ll}
                             2\zeta_{1}z_{1}+\tau_{1} t_{1}=0 & \hbox{ } \\
                              2\zeta_{i}z_{i}+\tau_{i} t_{i}+ \tau_{i-1} t_{i-1}=0,  & \hbox{$2\leq i \leq p-1$} \\
                              2\zeta_{p}z_{p}+\tau_{p-1} t_{p-1}=0. & \hbox{ }
                           \end{array}
                         \right.$$
On en déduit que
$$\mathfrak{g}(g)=\bigoplus_{i=1}^{p-1}\mathbb{K} W_{i},$$
avec $$W_{i}=T_{i}-\frac{\tau_{i}}{2}(\frac{1}{\zeta_{i}} Z_{i} + \frac{1}{\zeta_{i+1}} Z_{i+1})=E_{2i-1,2i+2}+E_{2i+1,2i}-\frac{\tau_{i}}{2}(\frac{1}{\zeta_{i}} E_{2i-1,2i}+ \frac{1}{\zeta_{i+1}} E_{2i+1,2i+2}).$$
Soit $H=\sum_{i=1}^{p} H_{i}$. Alors, on a
$$[H, Z_{l}]= 2 Z_{l} \,\,\,\, et\,\,\,\,[H, T_{l}]= 2 T_{l}$$
et donc  $$[H, W_{l}]=2 W_{l},\,\,\, 1\leq l \leq p-1.$$
Il en résulte que
$$[\mathfrak{g},\mathfrak{g}(g)]\cap \mathfrak{g}(g) \neq \{0\}.$$
Compte tenu du lemme \ref{Tauvel Yu} et du fait que $g$ est générique, il résulte  que $\mathfrak{g}$ n'admet pas de forme linéaire stable.\\
\\(b) $\dim V=2p+1$, $p\geq 2$ et $\dim V_{i}=2i-1$ pour tout  $1\leq i \leq t-1$ de sorte que $\mathcal{V}=\{ \{0\}=V_{0} \varsubsetneq V_{1} \varsubsetneq V_{2}\varsubsetneq...\varsubsetneq V_{p} \varsubsetneq V_{p+1}=V \}$. On choisit une base $(e_{1},...,e_{2p+1})$ de $V$ adaptée au drapeau $\mathcal{V}$ telle que $V_{j}=\mathbb{K}e_{1}\oplus...\oplus \mathbb{K}e_{2j-1},  \,\,\,1\leq j \leq p+1$ et $\xi=\sum_{j=1}^{p} e^{\ast}_{2j-1} \wedge e^{\ast}_{2j}$.\\
\\Dans ces conditions et compte tenu de \cite[5.6]{DKT}, les matrices des éléments  de $\mathfrak{r}_{\mathcal{V}}$ dans la base $e_{1},...,e_{2p+1}$ sont de la forme:\\

$\left(
\begin{array}{ccccccccccccc}
    a_{1} & z_{1} & 0 & t_{1} & 0 & 0 &  &  &  &  &  &  &  \\
    0 & -a_{1} & 0 & 0 & 0 & 0 &  &  & &  & & &  \\
    0 & t_{1} & a_{2} & z_{2} & 0 & t_{2}& & &  & &  &  &  \\
    0 & 0 & 0 & - a_{2} & 0 & 0 &  & &  & & &  & \\
    0 & 0 & 0 & t_{2} & a_{3} & z_{3} &  &  &  &  &  &  &  \\
    0 & 0 & 0 & 0 & 0 & - a_{3} &  &  &  &  &  &  &  \\
     &  &  &  &  &  & . &  &  &  &  &  &  \\
     &  &  &  &  &  &  & . &  &  &  &  &  \\
     &  &  &  &  &  &  &  & a_{p-1} & z_{p-1} & 0 & t_{p-1} & 0 \\
     &  &  &  &  &  &  &  & 0 & - a_{p-1} & 0 & 0 & 0 \\
     &  &  &  &  &  &  &  & 0 & t_{p-1} & a_{p} & z_{p} & 0 \\
     &  &  &  &  &  &  &  & 0 & 0 & 0 & - a_{p} & 0 \\
     &  &  &  &  & &  &  & 0 & 0 & 0 & t_{p} & a_{p+1} \\
\end{array}
\right)$
\\
\\On voit alors que si l'on pose $V^{\prime}=\mathbb{K}e_{1}\oplus...\oplus \mathbb{K}e_{2p}$,  et $\mathcal{V}^{\prime}=\{ \{0\}=V^{\prime}_{0} \varsubsetneq V^{\prime}_{1}=V_{1} \varsubsetneq V^{\prime}_{2}=V_{2}\varsubsetneq  V^{\prime}_{p}=V_{p} \varsubsetneq V^{\prime}_{p+1}=V^{\prime} \}$,  $\xi|_{V^{\prime}}$ est symplectique et $\mathcal{V}^{\prime}$ est en position générique par rapport \`{a} $\xi|_{V^{\prime}}$ tandis que
\begin{equation}
\mathfrak{r}_{\mathcal{V}}=\mathfrak{l}_{\mathcal{V}}\oplus \mathfrak{n}_{\mathcal{V}}
\end{equation}
o\`{u} $\mathfrak{l}_{\mathcal{V}}=\mathfrak{l}_{\mathcal{V}^{\prime}}\oplus \mathbb{K} H_{p+1}$ est un facteur réductif de $\mathfrak{r}_{\mathcal{V}}$ qui est un tore dans ce cas, et $\mathfrak{n}_{\mathcal{V}}=\mathfrak{n}_{\mathcal{V}^{\prime}}\oplus\mathbb{K} T_{p}$ son radical unipotent, avec
\begin{equation}
H_{p+1}=E_{2p+1,2p+1}\hbox{ et } T_{p}=E_{2p+1,2p}.
\end{equation}
Dans ce cas, on a les crochets supplémentaires
$$[H_{p+1}, Z_{j}]=0 \,\,\, 1\leq j  \leq p \,\,\, , [H_{p+1}, T_{j}]=0 \,\,\, 1\leq j  \leq p-1$$
$$[H_{j}, T_{p}]= \delta_{jp} T_{p} \,\,\, 1\leq j  \leq p \,\,\, , [H_{p+1}, T_{p}]=T_{p}. $$
Soit $g$ une forme fortement régulière sur $\mathfrak{g}=\mathfrak{r}_{\mathcal{V}}$. On pose $n=g|_{\mathfrak{n}_{\mathcal{V}}}$, $\zeta_{j}=g(Z_{j})$ et $\tau_{j}=g(T_{j})$.\\
Supposons que $\tau_{p}\prod_{j=1}^{p} \zeta_{j}\neq 0$. Alors, on a $\mathfrak{g}(n)=\mathfrak{n}_{\mathcal{V}}$ de sorte que, compte tenu du lemme \ref{duf82}, $\mathbf{N}_{\mathcal{V}}.g=g+ \mathfrak{n}_{\mathcal{V}}^{\perp}$, o\`{u} $\mathbf{N}_{\mathcal{V}}$ désigne le radical unipotent du sous-groupe algébrique connexe $\mathbf{R}_{\mathcal{V}}$ de $\mathbf{G}\mathbf{L}(V)$ d'algèbre de Lie $\mathfrak{r}_{\mathcal{V}}$, et qu'on peut supposer que $g|_{\mathfrak{l}_{\mathcal{V}}}=0$.\\
D'autre part, si  $X=\sum_{l=1}^{p}z_{l} Z_{l}+ \sum_{l=1}^{p}t_{l} T_{l}$,  on a
\begin{eqnarray}
\begin{split}
[H_{1}, X]  &=2 z_{1} Z_{1}+  t_{1} T_{1}\\
[H_{i}, X] &=2 z_{i} Z_{i}+  t_{i-1} T_{i-1}+ t_{i} T_{i},\,\,2\leq i \leq p,\\
[H_{p+1}, X]&= t_{p} T_{p}
\end{split}
\end{eqnarray}
Il s'ensuit que $X \in \mathfrak{g}(g)=\mathfrak{n}_{\mathcal{V}}(g)$ si et seulement si
$$\left\{
                           \begin{array}{ll}
                             2\zeta_{1}z_{1}+\tau_{1} t_{1}=0, & \hbox{} \\
                              2\zeta_{i}z_{i}+ \tau_{i-1} t_{i-1}+\tau_{i} t_{i}=0, & \hbox{$2\leq i \leq p-1$} \\
                              \tau_{p} t_{p}=0, & \hbox{.}
                           \end{array}
                         \right.$$
On en déduit que $\mathfrak{g}(g)=\bigoplus_{i=1}^{p-1}\mathbb{K} W_{i},$
avec   $$W_{i}=T_{i}-\frac{\tau_{i}}{2}(\frac{1}{\zeta_{i}} Z_{i} + \frac{1}{\zeta_{i+1}} Z_{i+1})=E_{2i-1,2i+2}+E_{2i+1,2i}-\frac{\tau_{i}}{2}(\frac{1}{\zeta_{i}} E_{2i-1,2i}+ \frac{1}{\zeta_{i+1}} E_{2i+1,2i+2}).$$
Si $H=\sum_{i=1}^{p+1} H_{i}$,  on a $$[H, W_{i}]=2 W_{i} \hbox{ pour } 1\leq i \leq p-1. $$
Par suite, on a $$[\mathfrak{g},\mathfrak{g}(g)]\cap \mathfrak{g}(g) \neq \{0\}.$$
Comme $g$ est générique, ceci achève la démonstration dans ce cas.
\end{proof}
\section{Cas des sous-algèbres paraboliques de $\mathfrak{s}\mathfrak{o}(E)$ }\label{première section}
Dans ce num\'{e}ro, on va \'{e}tudier la stabilit\'{e} des sous-alg\`{e}bres paraboliques des alg\`{e}bres de Lie orthogonales. Exactement, on va montrer l'assertion $(\text{ii})$ de la conjecture de Panyushev \cite[conjecture 5.6] {Panyushev2005} pour le cas des sous-alg\`{e}bres paraboliques d'une alg\`{e}bre de Lie orthogonale.\\

\subsection{Réduction au cas de rang nul pour les sous-algèbres paraboliques d'une algèbre de Lie orthogonale}\label{red.parab}
 le but de ce paragraphe est de ramener l'étude de la stabilité au cas de rang nul pour les sous-algèbres paraboliques et qui nous sera utile pour la suite:\\
Soit $E$ un espace vectoriel de dimension finie sur $\mathbb{K}$ muni d'une forme bilinéaire symétrique non dégénérée $B$. On désigne par $\mathfrak{s}\mathfrak{o}(E)$ l'algèbre de Lie du groupe orthogonal correspondant. Soit $\mathcal{V}=\{ \{0\}=V_{0} \varsubsetneq V_{1} \varsubsetneq V_{2}\varsubsetneq...\varsubsetneq V_{t-1} \varsubsetneq V_{t}=V \}$ un drapeau de sous-espaces isotropes de $E$ avec $\dim V\geq 1$. On désigne par $\mathfrak{p}_{\mathcal{V}}$ la sous-algèbre de Lie de $\mathfrak{s}\mathfrak{o}(E)$ constituée des endomorphismes de $E$ stabilisant le drapeau $\mathcal{V}$ : c'est une sous-algèbre parabolique de $\mathfrak{s}\mathfrak{o}(E)$ et les sous-algèbres paraboliques de $\mathfrak{s}\mathfrak{o}(E)$ sont toutes obtenues ainsi.\\
 Soit maintenant $\mathfrak{g}=\mathfrak{p}_{\mathcal{V}}$. Supposons que $\mathfrak{g}$ ne soit pas quasi-réductive et donc que $\rang \mathfrak{g} < \ind \mathfrak{g}$. On cherche \`{a} savoir si $\mathfrak{g}$ admet ou non des formes stables.\\
Soit $\mathfrak{j}$ une sous-algèbre de Cartan-Duflo de $\mathfrak{g}$ et $\Pi$ l'ensemble des poids non nuls de $\mathfrak{j}$ dans $E$. Si $\alpha \in \Pi \cup \{0\}$, on note $E^{\alpha}$ le sous-espace poids correspondant. Si $\alpha \in \Pi$, $-\alpha \in \Pi$, de plus $E^{\alpha}$ et $E^{-\alpha}$ sont des sous-espaces totalement isotropes mis en dualité par $B$. Si $\alpha,\beta \in \Pi$ sont tels que $\alpha+\beta \neq 0$, alors $E^{\alpha}$ et $E^{\beta}$ sont orthogonaux.\\
Soit $\Pi^{+}\subset \Pi$ une partie telle que $\Pi=\Pi^{+}\amalg \Pi^{-}$ o\`{u} $\Pi^{-}=-\Pi^{+}$.\\
Alors, on a
$$\mathfrak{s}\mathfrak{o}(E)^{\mathfrak{j}}=\mathfrak{s}\mathfrak{o}(E^{0})\times \prod_{\alpha \in \Pi^{+}}\mathfrak{g}\mathfrak{l}(E^{\alpha}).$$
Pour $\alpha \in \Pi \cup \{0\}$, soit $\mathcal{V}^{\alpha}$ le drapeau de $E^{\alpha}$ dont les éléments sont ceux
de $U\cap E^{\alpha}, U \in \mathcal{V}$ rangés dans l'ordre strictement croissant, auxquels on adjoint $E^{\alpha}$ si $\alpha \neq 0$. Alors, on a $$\mathfrak{g}^{\mathfrak{j}}=\mathfrak{p}_{\mathcal{V}^{0}}\times \prod_{\alpha \in \Pi^{+}}\mathfrak{q}_{\mathcal{V}^{\alpha}}$$
o\`{u} avec les notations du numéro précédent $\mathfrak{q}_{\mathcal{V}^{\alpha}}$ désigne la sous-algèbre parabolique de $\mathfrak{g}\mathfrak{l}(E^{\alpha})$ qui stabilise le drapeau $\mathcal{V}^{\alpha}$. D'autre part, d'après Panyushev (voir\cite{Panyushev2005}) les sous-algèbres paraboliques $\mathfrak{q}_{\mathcal{V}^{\alpha}}$ sont stables. Il s'ensuit que $\mathfrak{g}^{\mathfrak{j}}$ admet des formes stables si et seulement si $\mathfrak{p}_{\mathcal{V}^{0}}$ admet des formes stables. Or d'apr\`{e}s \ref{cons} et comme son centre est trivial, $\mathfrak{p}_{\mathcal{V}^{0}}$ est de $\mathrm{rang}$ nul. Ainsi et compte tenu de la proposition \ref{prop1}, pour démontrer l'équivalence entre la quasi-réductivité et la stabilité dans le cas considéré, on se ram\`{e}ne \`{a} montrer que  les sous-alg\`{e}bres paraboliques de $\mathfrak{s}\mathfrak{o}(E)$ qui sont de $\mathrm{rang}$ nul et d'indice strictement positif ne sont pas stables.

\begin{theo}\label{D.K.T}\emph{(Voir \cite[5.15]{DKT})}
Soit $E$ un espace vectoriel de dimension $q\geq 3$ sur $\mathbb{K}$ muni d'une forme bilinéaire symétrique non dégénérée $B$, $\mathcal{V}=\{ \{0\}=V_{0} \varsubsetneq V_{1} \varsubsetneq V_{2}\varsubsetneq...\varsubsetneq V_{t-1} \varsubsetneq V_{t}=V\}$ un drapeau de sous-espaces isotropes de $E$ avec $r=\dim V\geq 1$. On note $\mathcal{V}^{\prime}$ le drapeau tel que  $\mathcal{V}^{\prime}=\mathcal{V}\setminus\{V\}$, si $r$ est impair égal \`{a} $\frac{q}{2}$, et $\mathcal{V}^{\prime}=\mathcal{V}$, sinon.\\
$\text{(i)}$ On a les formules suivantes pour l'indice de la sous-algèbre parabolique $\mathfrak{p}_{\mathcal{V}}$:
$$\ind(\mathfrak{p}_{\mathcal{V}})=\left\{
  \begin{array}{llll}

    [\frac{q}{2}]-r + \sum_{i=1}^{t}[\frac{1}{2}(\dim V_{i}-\dim V_{i-1})] & \hbox{ si r est pair,} \\
\\

 [\frac{q-1}{2}]-r + \sum_{i=1}^{t}[\frac{1}{2}(\dim V_{i}-\dim V_{i-1})] & \hbox{ si r est impair, et $r< \frac{q}{2}$,} \\
\\
    \sum_{i=1}^{t}[\frac{1}{2}(\dim V_{i}-\dim V_{i-1})]-1 & \hbox{ si r est impair, $r=\frac{q}{2}$ et }\\

  &\hbox{$\dim V_{t-1}< r-1$,}\\
\\
    \sum_{i=1}^{t}[\frac{1}{2}(\dim V_{i}-\dim V_{i-1})]+1 & \hbox{ si r est impair, $r=\frac{q}{2}$ et }\\
 &\hbox{ $\dim V_{t-1}= r-1$, }
  \end{array}
\right.$$
$\text{(ii)}$ La dimension du radical unipotent d'un stabilisateur générique de la représentation coadjointe de $\mathfrak{p}_{\mathcal{V}}$ est $h(\mathcal{V}^{\prime})$.\\
 $\text{(iii)}$ L'algèbre de Lie $\mathfrak{p}_{\mathcal{V}}$ est quasi-réductive si et seulement si le drapeau $\mathcal{V}^{\prime}$ vérifie la propriété $\mathcal{P}$.
\end{theo}

\begin{pr}\label{}
Soit $\mathcal{V}=\{ \{0\}=V_{0} \varsubsetneq V_{1} \varsubsetneq V_{2}\varsubsetneq...\varsubsetneq V_{t-1}\varsubsetneq V_{t}=V \}$ un drapeau de sous-espaces isotropes de E et $\mathfrak{p}_{\mathcal{V}}$ la sous-algèbre parabolique qui stabilise le drapeau $\mathcal{V}$. On pose $q=\dim E$ et $r=\dim V$. Alors les assertions suivantes sont équivalentes:\\
$\text{(i)}$ $\mathfrak{p}_{\mathcal{V}}$ est de $\rang$ nul.\\
$\text{(ii)}$ la condition $(\ast)$ est satisfaite par $\mathcal{V}$ et $(q , r)$ vérifie l'une des conditions suivantes:
\begin{eqnarray}
\begin{split}
&(a) \  q \in \{2r,2r+1\}.\\
&(b)  \  r \hbox{ est impair et } q=2r+2.
\end{split}
\end{eqnarray}
\end{pr}

\begin{proof}:
Remarquons tout d'abord que l'algèbre de lie $\mathfrak{p}_{\mathcal{V}}$ est de \rang nul si et seulement si le nombre $h(\mathcal{V}^{\prime})$ est égal à l'indice de $\mathfrak{p}_{\mathcal{V}}$. D'après le théorème \ref{D.K.T}, il suffit de considérer les  cas suivants:\\
\\On suppose que $\dim V=r$ est pair de sorte que $\mathcal{V}^{\prime}=\mathcal{V}$. Dans ce cas, on a $$ \ind(\mathfrak{p}_{\mathcal{V}})= [\frac{q}{2}]-r + \sum_{i=1}^{t}[\frac{1}{2}(\dim V_{i}-\dim V_{i-1})].$$
Par suite, on a $$\rang \mathfrak{p}_{\mathcal{V}}=0 \hbox{ si et seulement si } [\frac{q}{2}]-r =0 \hbox{ et } \sum_{i=1}^{t}[\frac{1}{2}(\dim V_{i}-\dim V_{i-1})]=h(\mathcal{V}).$$
 Ainsi, $\mathfrak{p}_{\mathcal{V}}$ est de \rang nul si et seulement si  $q \in \{2r,2r+1\}$ et le drapeau $\mathcal{V}$ vérifie la condition $(\ast)$.\\
\\On suppose que $\dim V=r$ est impair et $r< \frac{q}{2}$ de sorte que $\mathcal{V}^{\prime}=\mathcal{V}$. Dans ce cas, on a $$ \ind(\mathfrak{p}_{\mathcal{V}})= [\frac{q-1}{2}]-r + \sum_{i=1}^{t}[\frac{1}{2}(\dim V_{i}-\dim V_{i-1})].$$
Par suite, on a $$\rang \mathfrak{p}_{\mathcal{V}}=0 \hbox{ si et seulement si } [\frac{q-1}{2}]-r =0 \hbox{ et } \sum_{i=1}^{t}[\frac{1}{2}(\dim V_{i}-\dim V_{i-1})]=h(\mathcal{V}).$$
 Ainsi, $\mathfrak{p}_{\mathcal{V}}$ est de \rang nul si et seulement si  $q \in \{2r+1,2r+2\}$ et le drapeau $\mathcal{V}$ vérifie la condition $(\ast)$.\\
\\On suppose maintenant que $\dim V=r$ est impair, $r= \frac{q}{2}$. Dans ces conditions et compte tenu du théorème \ref{D.K.T}, on a
$$\ind(\mathfrak{p}_{\mathcal{V}})=\left\{
                                    \begin{array}{ll}
                                     \sum_{i=1}^{t}[\frac{1}{2}(\dim V_{i}-\dim V_{i-1})]-1  & \hbox{ si $\dim V_{t-1}< r-1$ } \\
                                      \sum_{i=1}^{t}[\frac{1}{2}(\dim V_{i}-\dim V_{i-1})]+1 & \hbox{ si $\dim V_{t-1}= r-1$ }
                                    \end{array}
                                  \right.
$$
La proposition est alors claire dans ce cas.
\end{proof}
\subsection{}\label{}
On garde les notations du numéro précédent. Si $u,v \in E$, on définit l'élément $u\wedge_{B}v$ de $\mathfrak{s}\mathfrak{o}(E)$ en posant, pour $x \in E$:
$$u\wedge_{B}v (x)= B(x,u) v- B(x,v) u.$$
Il existe une unique application linéaire $\wedge_{B}:E\otimes E\mapsto \mathfrak{s}\mathfrak{o}(E)$ telle que $\wedge_{B}(u \otimes v)=u \wedge_{B} v$, laquelle passe au quotient en un isomorphisme d'espaces vectoriels de $\wedge^{2}E$ sur $\mathfrak{s}\mathfrak{o}(E)$. Plus généralement, si V, W sont des sous-espaces de $E$ tels que $V\cap W=\{0\}$, $\wedge_{B}$ induit un isomorphisme d'espaces vectoriels de $\wedge^{2}V$ ( resp. $V\otimes W$ ) sur un sous-espace de $\mathfrak{s}\mathfrak{o}(E)$ noté $\Lambda^{2}_{B}V$ (resp. $V \Lambda_{B} W$).\\
Si $u_{1}, u_{2}, v_{1}, v_{2} \in E$, on a
\begin{eqnarray}
\begin{split}
[u_{1}\wedge_{B} v_{1}, u_{2}\wedge_{B}v_{2}]= &B(u_{1},u_{2})v_{1}\wedge_{B} v_{2}+ B(u_{1},v_{2})u_{2}\wedge_{B} v_{1}\\
&+ B(v_{1},u_{2})u_{1}\wedge_{B} v_{2}+ B(v_{1},v_{2})u_{1}\wedge_{B} u_{2}.\\
\end{split}
\end{eqnarray}
On munit $\mathfrak{s}\mathfrak{o}(E)$ de la forme bilinéaire symétrique non dégénérée et invariante $L(X,Y)=-\frac{1}{2}Tr XY$. Si
$u_{1}, u_{2}, v_{1}, v_{2} \in E$, on a $$L(u_{1}\wedge_{B} v_{1},u_{2}\wedge_{B} v_{2})=det \left(
                                                                                                \begin{array}{cc}
                                                                                                  B(u_{1},u_{2}) & B(u_{1},v_{2}) \\
                                                                                                  B(v_{1},u_{2}) & B(v_{1},v_{2}) \\
                                                                                                \end{array}
                                                                                              \right).
$$
\subsubsection{}\label{not}
Soit $V$ un sous-espace isotrope  non nul de $E$ et $W$ un supplémentaire dans $E$ de l'orthogonal $V^{\bot_{B}}$ de V pour $B$. Alors $B$ induit une dualité entre $V$ et $W$, la restriction de $B$ \`{a} $F=(V\oplus W)^{\bot_{B}}$ est non dégénérée et on a $E=V\oplus F\oplus W$ de sorte que $V^{\bot_{B}}=V\oplus F$. On identifie les algèbres de Lie $\mathfrak{g}\mathfrak{l}(V)$ et $\mathfrak{s}\mathfrak{o}(F)$ \`{a} des sous-algèbres de Lie de $\mathfrak{s}\mathfrak{o}(E)$ de la manière suivante:\\
- si $X \in \mathfrak{g}\mathfrak{l}(V)$, on l'étend en un élément encore noté $X$ de $\mathfrak{s}\mathfrak{o}(E)$ en décidant que $X_{|F}=0$ et, si $y \in W$, $X.y$ est l'élément de $W$ tel que $$ B(X.x , y)+  B(x , X.y)=0, x \,\,\in V,$$
- si $X \in \mathfrak{s}\mathfrak{o}(F)$, on l'étend en un élément encore noté $X$ de $\mathfrak{s}\mathfrak{o}(E)$ en décidant que $X_{|F^{\bot_{B}}}=0$.\\
De plus, si $X \in \mathfrak{g}\mathfrak{l}(V)$, $Y \in \mathfrak{s}\mathfrak{o}(F)$, $u \in F$, $v,v^{\prime} \in V$, on a
\begin{eqnarray}
\begin{split}
[X, u\wedge_{B}v]&=u\wedge_{B}X.v,\\
[Y, u\wedge_{B}v]&=Y.u\wedge_{B}v,\\
[Y, v\wedge_{B}v^{\prime}]&=0,\\
[X, v\wedge_{B}v^{\prime}]&=X. v\wedge_{B}v^{\prime}+ v\wedge_{B}X.v^{\prime}.
\end{split}
\end{eqnarray}
- La forme $L$ permet d'identifier $\wedge_{2}^{B}V^{\ast}$ (rep.$(F\wedge_{B}V)^{\ast}$) avec $\wedge_{2}^{B}W$ (resp.$F\wedge_{B}V$).\\

Le théorème suivant est le résultat principal de ce travail.
\begin{theo}\label{stab de p}
Soit $E$ un espace vectoriel de dimension finie $q\geq 3$ sur $\mathbb{K}$ muni d'une forme bilinéaire symétrique $B$ non dégénérée  et $\mathcal{V}=\{ \{0\}=V_{0} \varsubsetneq V_{1} \varsubsetneq V_{2}\varsubsetneq...\varsubsetneq V_{t-1} \varsubsetneq V_{t}=V\}$ un drapeau de sous-espaces isotropes de $E$ avec $\dim V\geq 1$. On note $\mathfrak{p}_{\mathcal{V}}$ la sous-algèbre parabolique de $\mathfrak{s}\mathfrak{o}(E)$ qui stabilise le drapeau $\mathcal{V}$. Alors les assertions suivantes sont équivalentes:\\
$\text{(i)}$ $\mathfrak{p}_{\mathcal{V}}$ admet une forme linéaire stable.\\
$\text{(ii)}$ $\mathfrak{p}_{\mathcal{V}}$ est quasi-réductive.
\end{theo}
\begin{proof}
On pose $\mathfrak{g}=\mathfrak{p}_{\mathcal{V}}$ et $\dim V=r$. D'après \ref{red.parab}, on peut supposer que l'algèbre de Lie $\mathfrak{p}_{\mathcal{V}}$ est de \rang nul c'est à dire que $\mathcal{V}$ vérifie la condition $(\ast)$ et $(q,r)$ vérifie $(4.1)$. Il s'agit de montrer que si l'indice de $\mathfrak{p}_{\mathcal{V}}$ est strictement positif, alors $\mathfrak{p}_{\mathcal{V}}$ n'admet pas de formes stables.\\
On garde les notations du numéro précédent et on commence par étudier un cas particulier: on suppose que $t\geq 2$ et pour $1\leq i \leq t-1$, $\dim V_{i}$ est impair avec $\dim V_{i}=2i-1$, de sorte que $\mathcal{V}=\{ \{0\}=V_{0} \varsubsetneq V_{1} \varsubsetneq V_{2}\varsubsetneq...\varsubsetneq V_{p} \varsubsetneq V_{p+1}=V\}$, avec $t=p+1$ et $r\in \{2p,2p+1\}$.\\
D'après \cite[5.15]{DKT}, on a $$\mathfrak{g}=\left\{
                                 \begin{array}{ll}
                                   \mathfrak{q}_{\mathcal{V}}\oplus \mathfrak{n}_{V}& \hbox{ si $r=\dim V$ est pair ou $F=0$,} \\
                                   \mathfrak{q}_{\mathcal{V}}\times \mathfrak{s}\mathfrak{o}(F) \oplus \mathfrak{n}_{V}& \hbox{ si $r=\dim V$ est impair et $F\neq0$.}
                                 \end{array}
                               \right.$$
avec $\mathfrak{n}_{V}=F\wedge_{B} V \oplus \mathfrak{z}_{V}$, $\mathfrak{z}_{V}=\Lambda^{2}V$ (le centre de $\mathfrak{n}_{V}$), o\`{u} $\mathfrak{n}_{V}$ est un idéal unipotent de $\mathfrak{g}$ et $\mathfrak{q}_{\mathcal{V}}$ la sous-algèbre parabolique de $\mathfrak{g}\mathfrak{l}(V)$ qui stabilise le drapeau $\mathcal{V}$.\\
Soit $g$ une forme linéaire générique sur $\mathfrak{g}$ et $n$ (resp. $\xi$) sa restriction \`{a} $\mathfrak{n}_{V}$ (resp. $\mathfrak{z}_{V}$). Par généricité de $g$, on peut supposer que $\xi$ est une forme bilinéaire alternée  sur $V$ générique relativement au drapeau $\mathcal{V}$. On choisit une base $e_{1},...,e_{r}$ de $V$ adaptée au drapeau $\mathcal{V}$ telle que, notant $e_{1}^{\ast},...,e_{r}^{\ast}$ la base duale, on ait $\xi= \sum_{1\leq2i \leq r}e_{2i-1}^{\ast} \wedge e_{2i}^{\ast}$. Compte tenu  des numéros 5.16, 5.18 et 5.19 de \cite{DKT}, on peut supposer que $n|_{F\wedge_{B} V}=0$ si $r$ est pair ou $F=0$ et $n=\xi+  u^{\ast}\wedge e_{r}^{\ast}$, avec $u \in F$ non isotrope, si $r$ est impair et $F\neq0$. On note $(E_{ij})_{1\leq i,j \leq r}$ la base de $\mathfrak{g}\mathfrak{l}(V)$ telle que $E_{ij}(e_{k})=\delta_{j,k} e_{i}$, $1\leq k \leq r$.
\\Soit $\mathfrak{r}_{\mathcal{V}}$ la sous-algèbre de Lie de $\mathfrak{q}_{\mathcal{V}}$ qui stabilise la forme bilinéaire $\xi$ sur $V$. Compte tenu des formules $(3.2)$, $(3.7)$ et $(3.8)$ de la démonstration du théorème \ref{stab de r} dont on reprend les notations, on a
\begin{eqnarray}
\begin{split}
\mathfrak{r}_{\mathcal{V}}=\mathfrak{l}_{\mathcal{V}}\oplus \mathfrak{n}_{\mathcal{V}}.
\end{split}
\end{eqnarray}
avec $$ \mathfrak{l}_{\mathcal{V}}= \bigoplus_{j=1}^{s} \mathbb{K} H_{j}\,\,et \,\, \mathfrak{n}_{\mathcal{V}}= (\bigoplus_{j=1}^{p} \mathbb{K} Z_{j})\oplus (\bigoplus_{j=1}^{s-1} \mathbb{K} T_{j}),$$
o\`{u} $s=p$ si $r=2p$ et $s=p+1$ si $r=2p+1$.\\
Soit alors $$ \mathfrak{m}_{\mathcal{V}} =\mathfrak{m}_{\mathcal{V},0} \oplus \mathfrak{m}_{\mathcal{V},1}\oplus \mathfrak{m}_{\mathcal{V},2},$$
le supplémentaire de $\mathfrak{r}_{\mathcal{V}}$ dans $\mathfrak{q}_{\mathcal{V}}$ tel que
\begin{eqnarray}
\begin{split}
 \mathfrak{m}_{\mathcal{V},0}&=\bigoplus_{i=1}^{p}\mathbb{K}S_{i}\,\,\,\,\,\, \hbox{ avec } S_{i}=E_{2i-1,2i-1}+E_{2i,2i},\\
\mathfrak{m}_{\mathcal{V},1}&=\bigoplus_{i=1}^{s-1}\mathbb{K} E_{2i,2i+1},\\
\mathfrak{m}_{\mathcal{V},2}&=\bigoplus_{j>i+1}\mathbb{K} E_{i,j}.\\
\end{split}
\end{eqnarray}
D'une part et compte tenu  des numéros 5.16 et 5.18  de \cite{DKT}, si  $r=2p$ ou $F=0$,  on a
$$\mathfrak{h}:=\mathfrak{g}(n)=\mathfrak{r}_{\mathcal{V}}\oplus \mathfrak{z}_{V}.$$
D'autre part et compte tenu du numéro 5.19 de \cite{DKT}, si  $r=2p+1$  et $F\neq0$, on a
$$\mathfrak{h}:=\mathfrak{g}(n)=\mathfrak{r}_{\mathcal{V}^{\prime}} \oplus \mathfrak{t}_{V^{\prime}} \oplus \mathfrak{n}_{V}(n),$$
o\`{u} $\mathcal{V}^{\prime}=\{ \{0\}=V^{\prime}_{0} \varsubsetneq V^{\prime}_{1}=V_{1} \varsubsetneq V^{\prime}_{2}=V_{2}\varsubsetneq...\varsubsetneq  V^{\prime}_{p}=V_{p} \varsubsetneq V^{\prime}_{p+1}=V^{\prime} \}$ est le drapeau obtenu  en remplaçant l'espace $V$ par $V^{\prime}$ ($\xi|_{V^{\prime}}$ est symplectique et $\mathcal{V}^{\prime}$  est en position générique par rapport \`{a} $\xi|_{V^{\prime}}$), $\mathfrak{t}_{V^{\prime}}= \mathbb{K} (T_{p}- u \wedge_{B}e_{2p-1} )$ et $\mathfrak{n}_{V}(n)=F\wedge_{B}e_{2p+1} \oplus \mathfrak{z}_{V}$. De plus, on a $\mathfrak{h}+ \mathfrak{n}_{V}=\mathfrak{r}_{\mathcal{V}^{\prime}}\oplus \mathbb{K} T_{p} \oplus \mathfrak{n}_{V}$.\\
Quitte alors \`{a} translater $g$ par un élément de $exp\,\, \mathfrak{n}_{V}$, on peut supposer que $g|_{\mathfrak{m}_{\mathcal{V}}}=0$ dans les deux cas traités.\\
\\On pose $\zeta_{i}=g(Z_{i})$, $1 \leq i \leq p$, $\tau_{i}=g(T_{i})$, $ 1 \leq i \leq s-1$.
Par généricité de $g$, on peut supposer que $\prod_{i=1}^{p} \zeta_{i}\neq 0$ si $r$ est pair et  $\tau_{p} \prod_{i=1}^{p} \zeta_{i}\neq 0 $ si $r$ est impair. Soit $\lambda=g|_{\mathfrak{r}_{\mathcal{V}}}$. Dans ces conditions et compte tenu de la démonstration du théorème \ref{stab de r}, on peut également supposer que $g|_{\mathfrak{l}_{\mathcal{V}}}=\lambda|_{\mathfrak{l}_{\mathcal{V}}}=0$. De plus et comme $\lambda$ est générique, il suit des calculs  de la démonstration du théorème \ref{stab de r} que $$\mathfrak{r}_{\mathcal{V}}(\lambda)=\bigoplus_{k=1}^{p-1}\mathbb{K} W_{k},$$
avec $$W_{k}=E_{2k-1,2k+2}+E_{2k+1,2k}-\frac{\tau_{k}}{2}(\frac{1}{\zeta_{k}} E_{2k-1,2k}+ \frac{1}{\zeta_{k+1}} E_{2k+1,2k+2}),\,\,1 \leq k \leq p-1.$$
Avec ces hypothèses sur la forme linéaire $g$ et ces notations, on a les résultats suivants:
\begin{lem}\label{lem.theo}
On suppose que $r=2p$ ou $F=0$. Alors, on a\\

$\text{(i)}$\,\, $\mathfrak{g}(g)=\bigoplus_{k=1}^{p-1}\mathbb{K} C_{k}$, avec $$C_{k}=W_{k}+\frac{1}{2}\tau_{k} (e_{2k-1}\wedge e_{2k}-e_{2k+1}\wedge e_{2k+2})+\frac{1}{2}\zeta_{k+1} e_{2k-1}\wedge e_{2k+2}+\frac{1}{2}\zeta_{k} e_{2k}\wedge e_{2k+1}.$$

$\text{(ii)}$\,\,$[\mathfrak{g},\mathfrak{g}(g)]\cap \mathfrak{g}(g) \neq \{0\}$.
\end{lem}
\begin{proof}
On commence par montrer le lemme pour le cas d'un parabolique de $\mathfrak{s}\mathfrak{o}(8, \mathbb{K})$ qui nous sera utile pour le cas général: Dans \cite{Tauvel-Yu2004}, Tauvel et Yu ont donné un exemple d'une sous-algèbre parabolique minimale de $\mathfrak{s}\mathfrak{o}(8, \mathbb{K})$ non quasi-réductive qui n'admet aucune forme linéaire stable. Cette sous-algèbre parabolique est la sous-algèbre de $\mathfrak{s}\mathfrak{o}(8, \mathbb{K})$, qui stabilise le drapeau formé de sous-espaces isotropes suivant
$$\mathcal{V}=\{\{0\}\varsubsetneq V_{1}\varsubsetneq V_{2}\varsubsetneq V_{3}=V\},\hbox{ avec } \dim V_{1}=1, \dim V_{2}=3 \hbox{ et } \dim V_{3}=\dim V=4.$$
Dans ce cas, on a $\dim V=4$ est pair et $F=0$ de sorte que $\mathfrak{p}_{\mathcal{V}}=\mathfrak{q}_{\mathcal{V}}\oplus \mathfrak{z}_{V}$.\\
Il résulte du théorème \ref{D.K.T} que $\mathfrak{p}_{\mathcal{V}}$ est une algèbre de Lie non quasi-réductive d'indice 1. Soit $\{e_{1},e_{2},e_{3},e_{4}\}$ une base de $V$ adaptée au drapeau $\mathcal{V}$ telle que, notant $e_{1}^{\ast},e_{2}^{\ast},e_{3}^{\ast},e_{4}^{\ast}$ la base duale, on ait $\xi=e_{1}^{\ast}\wedge e_{2}^{\ast} + e_{3}^{\ast}\wedge e_{4}^{\ast}$.\\
Dans ce cas, on a $\mathfrak{r}_{\mathcal{V}}= \mathfrak{l}_{\mathcal{V}}\oplus \mathfrak{n}_{\mathcal{V}}$, avec
\begin{eqnarray*}
\begin{split}
 \mathfrak{l}_{\mathcal{V}}&=\mathbb{K} H_{1}\oplus\mathbb{K} H_{2}, \hbox{ o\`{u} } H_{1}=E_{11}-E_{22} \hbox{ et } H_{2}=E_{33}-E_{44},\\   \mathfrak{n}_{\mathcal{V}}&=\mathbb{K} Z_{1}\oplus\mathbb{K} Z_{2}\oplus\mathbb{K} T, \hbox{ o\`{u} } Z_{1}=2E_{12},\, Z_{2}= 2E_{34} \hbox{ et }T=E_{14}+E_{32},
\end{split}
\end{eqnarray*}
et $$\mathfrak{m}_{\mathcal{V}}=\mathbb{K}(E_{11}+E_{22})\oplus\mathbb{K}(E_{33}+E_{44})\oplus \mathbb{K}E_{23}\oplus(\oplus_{4\geq j > i+1} \mathbb{K}E_{ij}).$$
Compte tenu de nos hypothèses sur la forme linéaire $g$, on peut supposer que $g|_{\mathfrak{m}_{\mathcal{V}}}=g|_{\mathfrak{l}_{\mathcal{V}}}=0$. D'autre part, on a $$\mathfrak{g}(g)\subset \mathfrak{r}_{\mathcal{V}}(\lambda)\oplus \mathfrak{z}_{V}=\mathfrak{h}(h):=\mathfrak{g}(g)+\mathfrak{n}_{V}(n)=\mathfrak{g}(g)+ \mathfrak{z}_{V} \hbox{ o\`{u} } h=g|_{\mathfrak{h}}.$$
Soit $X \in \mathfrak{h}(h) $; il s'écrit sous la forme $$X= \alpha\,\, W + \sum_{1\leq i < j\leq 4} \,z_{ij}\,\, e_{i}\wedge e_{j}=\alpha\,\, (T-\frac{\tau}{2}(\frac{1}{\zeta_{1}} Z_{1} + \frac{1}{\zeta_{2}} Z_{2})) + \sum_{1\leq i < j\leq 4}\,\, z_{ij}\,\, e_{i}\wedge e_{j}.$$
On a $\dim\,\mathfrak{g}(g)=h(\mathcal{V})=1$ de sorte que les $z_{ij}$ sont entièrement déterminés par $\alpha$.\\
Compte tenu de ce qui précède, on a
$$                                   \left\{
                                          \begin{array}{ll}
                                            \langle g, [X,Y]  \rangle=0 & \hbox{$ \forall \,\, Y \,\, \in \mathfrak{z}_V$,} \\                                             \langle g, [X,Y] \rangle=0 & \hbox{$ \forall \,\, Y\,\,  \in \mathfrak{r}_{\mathcal{V}}$,}
                                          \end{array}\
                                        \right.$$
de sorte que  $$X \in \mathfrak{g}(g) \hbox{ si et seulement si } \langle g, [X,Y] \rangle=0, \,\,\,\forall \,\,\,Y\in \mathfrak{m}_{\mathcal{V}}.$$
Ainsi, on vérifie aisément que $X \in \mathfrak{g}(g)$  si et seulement si $$\left\{
                                                           \begin{array}{ll}
                                                             z_{12}=\frac{1}{2}\,\alpha\, \tau & \hbox{} \\
                                                              z_{14}=\frac{1}{2} \, \alpha \,  \zeta_{2} & \hbox{} \\
                                                              z_{23}=\frac{1}{2} \, \alpha \,  \zeta_{1} & \hbox{} \\
                                                              z_{34}=-\frac{1}{2}\,\alpha \,\tau & \hbox{} \\
                                                             z_{13}=z_{24}=0.& \hbox{}
                                                           \end{array}
                                                         \right.
$$
Par suite, on a $$\mathfrak{g}(g)=\mathbb{K}\, C,$$
avec
\begin{equation}C=T-\frac{\tau}{2}(\frac{1}{\zeta_{1}} Z_{1} + \frac{1}{\zeta_{2}} Z_{2})+\frac{1}{2}\,\, \tau\,\, (e_{1}\wedge e_{2}-e_{3}\wedge e_{4})+\frac{1}{2}\,\, \xi_{2}\,\, e_{1}\wedge e_{4}+\frac{1}{2}\,\, \xi_{1}\,\, e_{2}\wedge e_{3}.
\end{equation}
Mais alors, on remarque que $[E_{11}+E_{33}, C]=C$ de sorte que $[\mathfrak{g},\mathfrak{g}(g)]\cap \mathfrak{g}(g) \neq \{0\}$.\\
Maintenant on va montrer le lemme dans la cas général. En nous inspirant de $(4.6)$ on définit, pour $1\leq k \leq p-1$, les vecteurs $C_{k}$ en posant $$C_{k}=W_{k}+\frac{1}{2}\tau_{k}\, (e_{2k-1}\wedge e_{2k}-e_{2k+1}\wedge e_{2k+2})+\frac{1}{2}\zeta_{k+1}\, e_{2k-1}\wedge e_{2k+2}+\frac{1}{2}\zeta_{k}\, e_{2k}\wedge e_{2k+1}.$$
Remarquons tout d'abord que $C_{k}\in \mathfrak{h}(h)$ avec $h=g|_{\mathfrak{h}}$, de sorte que $$C_{k} \in \mathfrak{g}(g) \hbox{ si et seulement si } \langle g, [C_{k},Y] \rangle=0, \,\,\,\forall \,\,\,Y\in \mathfrak{m}_{\mathcal{V}_{j}}, j=0,1,2.$$
Compte tenu de $(4.5)$, on a
$$C_{k}\in \mathfrak{g}(g) \hbox{ si et seulement si }
\left\{
       \begin{array}{ll}
              \,\, \langle g, [S_{j},C_{k}]  \rangle=0,          &\hbox{$ 1\leq j\leq p$,} \\
               \,\, \langle g, [E_{2j,2j+1},C_{k}]  \rangle=0,    & \hbox{$1\leq j\leq s-1$,} \\
              \,\,  \langle g, [E_{i,j},C_{k}]  \rangle=0,        & \hbox{$2\leq i+1 < j \leq r$.}
      \end{array}
\right.$$
D'une part, on a
\begin{eqnarray*}
\begin{split}
[S_{j},C_{k}]&=\delta_{k,j} ( E_{2k-1,2k+2}-E_{2k+1,2k}
+\tau_{k}\,e_{2k-1}\wedge e_{2k}+\frac{1}{2}\zeta_{k+1}\,\,e_{2k-1}\wedge e_{2k+2}\\
&+\frac{1}{2}\zeta_{k}\,e_{2k}\wedge e_{2k+1})+\delta_{k+1,j} (- E_{2k-1,2k+2}+E_{2k+1,2k}-\tau_{k}\,e_{2k+1}\wedge e_{2k+2}\\
&+\frac{1}{2}\zeta_{k+1}\,e_{2k-1}\wedge e_{2k+2}+\frac{1}{2}\zeta_{k}\,e_{2k}\wedge e_{2k+1})\\
&=\delta_{k,j}\, (-T_{k}+ \tau_{k}\,e_{2k-1}\wedge e_{2k})+ \delta_{k+1,j}\, (T_{k}- \tau_{k}\,e_{2k+1}\wedge e_{2k+2})\, \textrm{mod}(\mathfrak{m}_{\mathcal{V}}+\ker(n)),
\end{split}
\end{eqnarray*}
de sorte que, compte tenu de nos hypothèses sur la forme linéaire $g$, $$\langle g, [S_{j},C_{k}]  \rangle=0,\,\,\hbox{ pour tout } 1\leq j\leq p.$$
D'autre part, on a
\begin{eqnarray*}
\begin{split}
[E_{2j,2j+1},C_{k}]&=\delta_{k-1,j} ( E_{2k-2,2k+2}- \frac{\tau_{k}}{2\zeta_{k}}E_{2k-2,2k}+ \frac{1}{2}\,\tau_{k}\,e_{2k-2}\wedge e_{2k}\\
&+\frac{1}{2}\zeta_{k+1}\,\,e_{2k-2}\wedge e_{2k+2})
+\delta_{k,j} (- E_{2k+1,2k+1}+E_{2k,2k}\\
&+  \frac{\tau_{k}}{2\zeta_{k}}E_{2k-1,2k+1}
- \frac{\tau_{k}}{2\zeta_{k+1}}E_{2k,2k+2}- \frac{1}{2} \tau_{k}\,e_{2k}\wedge e_{2k+2})\\
&+\delta_{k+1,j} (- E_{2k-2,2k+3}+ \frac{\tau_{k}}{2\zeta_{k+1}}E_{2k+1,2k+3})\\
&=\delta_{k,j} (H_{k}-H_{k+1}) \,\textrm{mod}( \mathfrak{m}_{\mathcal{V}}+\ker(n))
\end{split}
\end{eqnarray*}
comme on a $g(H_{k+1})=g(H_{k})=0$, il suit de nos hypothèses sur la forme linéaire $g$, que $$\langle g, [E_{2j,2j+1},C_{k}]  \rangle=0,\,\,\hbox{ pour tout } 1\leq j\leq s-1.$$
\\Enfin, pour calculer le crochet $[E_{i,j},C_{k}]$ avec $2\leq i+1 < j \leq r$, on distingue suivant la parité de $i$ et $j$:\\
(i) Pour  $i=2l$ et $j=2l^{\prime}$ avec $l+1\leq l^{\prime}$, on a \\
\begin{eqnarray*}
\begin{split}
[E_{2l,2l^{\prime}},C_{k}]&=\delta_{k+1,l} (- E_{2k-1,2l^{\prime}}+\frac{\tau_{k}}{2\zeta_{k+1}}\, E_{2k+1,2l^{\prime}})+\delta_{k,l}\,(-E_{2k+1,2l^{\prime}}+\frac{\tau_{k}}{2\zeta_{k}}  E_{2k-1,2l^{\prime}})\\
&+ \frac{1}{2}\,\,\delta_{k,l^{\prime}}(\tau_{k}\,e_{2k-1}\wedge e_{2l}+\zeta_{k}\,\,e_{2l}\wedge e_{2k+1})\\
&+\frac{1}{2}\,\,\delta_{k+1,l^{\prime}}(-\tau_{k}\,e_{2k+1}\wedge e_{2l}+\zeta_{k+1}\,\,e_{2k-1}\wedge e_{2l})\\
&=\delta_{k,l} \delta_{k+1,l^{\prime}}(-\frac{1}{2}\, Z_{k+1}+ \frac{1}{2}\,\zeta_{k+1}\,e_{2k-1}\wedge e_{2k})\,\textrm{mod}( \mathfrak{m}_{\mathcal{V}}+\ker(n)),\\
\end{split}
\end{eqnarray*}
de sorte que $$\langle g, [E_{2l,2l^{\prime}},C_{k}]  \rangle=0,\,\,\hbox{ pour tout } 1\leq l< l^{\prime}.$$
(ii) Pour  $i=2l$ et $j=2l^{\prime}+1$ avec $1\leq l <l^{\prime}$, on a\\
\begin{eqnarray*}
\begin{split}
[E_{2l,2l^{\prime}+1},C_{k}]&=\delta_{k+1,l} ( -E_{2k-1,2l^{\prime}+1}+\frac{\tau_{k}}{2\zeta_{k+1}}\, E_{2k+1,2l^{\prime}+1})\\
&+\delta_{k,l}\,(-E_{2k+1,2l^{\prime}+1}+\frac{\tau_{k}}{2\zeta_{k}}  E_{2k-1,2l^{\prime}+1})\\
&+\delta_{k-1,l^{\prime}}(E_{2l,2k+2}-\frac{\tau_{k}}{2\zeta_{k}}\, E_{2l,2k} +\frac{1}{2}\,\tau_{k}\,e_{2l}\wedge e_{2k}+\frac{1}{2}\,\zeta_{k+1}\,e_{2l}\wedge e_{2k+2})\\
&+\delta_{k,l^{\prime}}(E_{2l,2k}-\frac{\tau_{k}}{2\zeta_{k+1}}\, E_{2l,2k+2} -\frac{1}{2}\,\tau_{k}\,e_{2l}\wedge e_{2k+2}+\frac{1}{2}\,\zeta_{k}\,e_{2k}\wedge e_{2l}).\\
\end{split}
\end{eqnarray*}
Remarquons que, pour $1\leq l < l^{\prime}$, $[E_{2l,2l^{\prime}+1},C_{k}] \in \ker(n)+\mathfrak{m}_{\mathcal{V}}$ de sorte que $$\langle g, [E_{2l,2l^{\prime}+1},C_{k}]  \rangle=0.$$
(iii) Pour  $i=2l+1$ et $j=2l^{\prime}$ avec $1\leq l+1< l^{\prime}$, on a
\begin{eqnarray*}
\begin{split}
[E_{2l+1,2l^{\prime}},C_{k}]&= \frac{1}{2}\,\,\delta_{k,l^{\prime}}(\tau_{k}\,e_{2k-1}\wedge e_{2l+1}+\zeta_{k}\,\,e_{2l+1}\wedge e_{2k+1})\\
&+\frac{1}{2}\,\,\delta_{k+1,l^{\prime}}(-\tau_{k}\,e_{2k+1}\wedge e_{2l+1}+\zeta_{k+1}\,\,e_{2k-1}\wedge e_{2l+1}).\\
\end{split}
\end{eqnarray*}
On voit donc que, pour $1\leq l+1 < l^{\prime}$,  $[E_{2l,2l^{\prime}+1},C_{k}] \in \ker(n)$. Par suite $$\langle g, [E_{2l,2l^{\prime}+1},C_{k}]  \rangle=0.$$
(iv) Pour  $i=2l+1$ et $j=2l^{\prime}+1$ avec $0\leq l <l^{\prime}$, on a \\
\begin{eqnarray*}
\begin{split}
[E_{2l+1,2l^{\prime}+1},C_{k}]&=\delta_{k-1,l^{\prime}} ( E_{2l+1,2k+2}-\frac{\tau_{k}}{2\zeta_{k}}\, E_{2l+1,2k}+\frac{1}{2}\,\tau_{k}\,e_{2l+1}\wedge e_{2k}\\
&+\frac{1}{2}\,\zeta_{k+1}\,e_{2l+1}\wedge e_{2k+2})
+\delta_{k,l^{\prime}}(E_{2l+1,2k}-\frac{\tau_{k}}{2\zeta_{k+1}}\, E_{2l+1,2k+2}\\
&-\frac{1}{2}\,\tau_{k}\,e_{2l+1}\wedge e_{2k+2}-\frac{1}{2}\,\zeta_{k}\,e_{2l+1}\wedge e_{2k})\\
&=\delta_{k-1,l} \delta_{k,l^{\prime}}(\frac{1}{2}\, Z_{k}- \frac{1}{2}\,\zeta_{k}\,e_{2k-1}\wedge e_{2k})\,\textrm{mod}( \mathfrak{m}_{\mathcal{V}}+\ker(n)),
\end{split}
\end{eqnarray*}
de sorte que, compte tenu de nos hypothèses,   $$\langle g, [E_{2l+1,2l^{\prime}+1},C_{k}]  \rangle=0 \hbox{ pour tout } 0\leq l< l^{\prime}.$$
Par suite, pour $1\leq k \leq p-1$, on a $$\langle g, [C_{k},Y] \rangle=0 \,\,\,\forall \,\,\,Y\in \mathfrak{g}.$$
On a donc  $$ C_{k}\in \mathfrak{g}(g) \hbox{ pour tout } 1\leq k \leq p-1.$$
D'autre part, remarquons que la famille $\{C_{k}, 1\leq k \leq p-1\}$ est une famille libre de $\mathfrak{g}(g)$. De plus, comme $\dim(\mathfrak{g}(g))=h(\mathcal{V})=p-1$, cette famille est maximale et donc elle forme une base de $\mathfrak{g}(g)$. On a donc $$\mathfrak{g}(g)=\bigoplus_{k=1}^{p-1}\mathbb{K} C_{k}.$$
On pose $R_{k}=E_{2k-1,2k-1}+E_{2k+1,2k+1}$. On a $R_{k}\in\mathfrak{g}$ et
$$[R_{k},C_{k}]=C_{k} \hbox{ pour } 1\leq k\leq p-1,$$
ceci montre alors que $[\mathfrak{g},\mathfrak{g}(g)]\cap \mathfrak{g}(g) \neq \{0\}$ de sorte que $\mathfrak{g}$ est non stable puisque $g$ est une forme linéaire générique.
\end{proof}
On suppose maintenant que $r=\dim V=2p+1$ et $F\neq \{0\}$. Ici $\dim F\in\{1,2\}$, $\mathcal{V}=\{ \{0\}=V_{0} \varsubsetneq V_{1} \varsubsetneq V_{2}\varsubsetneq...\varsubsetneq V_{p} \varsubsetneq V_{p+1}=V\}$ et $n=\xi +  u^{\ast}\wedge e_{2p+1}^{\ast}$, avec $u \in F$ non isotrope. On rappelle les vecteurs $C_{k}$ du lemme \ref{lem.theo}.\\
\begin{lem}
On suppose que $r=\dim V$ est impair et $F\neq0$. Alors, on a\\

\text{(i)} $\mathfrak{g}(g)=\bigoplus_{k=1}^{p-1}\mathbb{K} C_{k} \oplus \mathbb{K} D_{p}$, avec

$$D_{p}=T_{p}- u\wedge_{B} e_{2p-1}-\tau_{p} u\wedge_{B}e_{2p+1}-\frac{\tau_{p}}{2\zeta_{p}}\, Z_{p}+\frac{1}{2}\,\,\tau_{p}\, e_{2p-1}\wedge e_{2p}+\frac{1}{2}\, \zeta_{p}\, e_{2p}\wedge e_{2p+1}.$$

\text{(ii)} $[\mathfrak{g},\mathfrak{g}(g)]\cap \mathfrak{g}(g) \neq \{0\}$.
\end{lem}
\begin{proof}
On va traiter tout d'abord les deux  exemples de \rang nul et d'indice 1 en dimensions $7$ et $8$ qui nous sera utile pour le cas général. On suppose que $\dim E=7$.  Ici $\dim F=1$, $\mathcal{V}=\{ \{0\}=V_{0} \varsubsetneq V_{1} \varsubsetneq V_{2}=V \}$, $\dim V_{1}=1$ et $\dim V_{2}=\dim V=3$. Soit $\{e_{1},e_{2},e_{3}\}$ une base de $V$ adaptée au drapeau $\mathcal{V}$ telle que, notant $e_{1}^{\ast},e_{2}^{\ast},e_{3}^{\ast}$ la base duale, on ait $\xi=e_{1}^{\ast}\wedge e_{2}^{\ast}$ et $n=\xi + u^{\ast}\wedge e_{3}^{\ast}$.\\
Dans ce cas, on a $$\mathfrak{r}_{\mathcal{V}}= \mathfrak{l}_{\mathcal{V}}\oplus \mathfrak{n}_{\mathcal{V}},$$ avec
\begin{eqnarray*}
\begin{split}
 \mathfrak{l}_{\mathcal{V}}&=\mathbb{K} H_{1}\oplus\mathbb{K} H_{2},\hbox{ o\`{u} } H_{1}=E_{11}-E_{22} \hbox{ et } H_{2}=E_{33}\\
 \mathfrak{n}_{\mathcal{V}}&=\mathbb{K} Z\oplus\mathbb{K} T,\hbox{ o\`{u} } Z=2E_{12} \hbox{ et }T=E_{32},
\end{split}
\end{eqnarray*}
et  $$\mathfrak{m}_{\mathcal{V}}=\mathbb{K}(E_{11}+E_{22})\oplus\mathbb{K}E_{23}\oplus \mathbb{K}E_{13}.$$
On pose  $g(Z)=\zeta$  et $g(T)= \tau$. Dans ce cas, on a
 $$\mathfrak{h}= \mathbb{K} H_{1}\oplus\mathbb{K}Z\oplus \mathfrak{o}_{V}\oplus \mathfrak{n}_{V}(n),$$
avec $\mathfrak{o}_{V}=\mathbb{K} (T- u\wedge_{B}e_{1})$ et $\mathfrak{n}_{V}(n)= \mathbb{K}\, u\wedge_{B}e_{3}\oplus \mathfrak{z}_{V}$.\\
\\D'après ce qui précède, tout élément $X$ de $\mathfrak{h}$ s'écrit sous la forme $$X= \alpha_{1}\,H_{1} +\alpha_{2}\,Z+ \alpha_{3}\,(T- u\wedge_{B}e_{1})+ \alpha_{4}\, u \wedge_{B}e_{3} + \sum_{1\leq i < j\leq 3} \,z_{ij}\,\, e_{i}\wedge e_{j}.\,\,\, (\star)$$
Si $X$ s'écrit comme en $(\star)$, on a $$ X \in \mathfrak{g}(g) \hbox{ si et seulement si } \left\{
                                          \begin{array}{ll}
                                            \langle g, [X,Y]  \rangle=0, & \hbox{$ \forall \,\, Y \,\, \in \mathfrak{r}_{\mathcal{V}}$} \\
                                             \langle g, [X,Y] \rangle=0 , & \hbox{$ \forall \,\, Y\,\,  \in \mathfrak{m}_{\mathcal{V}}$.}
                                          \end{array}\
                                        \right.$$
D'une part, on vérifie aisément que $\langle g, [X,\mathfrak{r}_{\mathcal{V}}]  \rangle=\{0\}$ si et seulement si $$ X=\alpha\,(T- u\,\,\wedge_{B}e_{1}-\tau\,\,u\,\, \wedge_{B}e_{3}-\frac{\tau}{2\zeta} Z)+\sum_{1\leq i< j\leq 3}z_{ij} \,\,e_{i}\wedge e_{j},\,\alpha \in \mathbb{K}.\,\,\, (\star\star)$$
D'autre part, on sait que $\ind(\mathfrak{g})=\dim\,\mathfrak{g}(g)=1$ de sorte que les $z_{ij}$ sont entièrement déterminés par $\alpha$.
Si X s'écrit comme en $(\star\star)$, on a $$ X \in \mathfrak{g}(g) \hbox{ si et seulement si }\langle g, [X,Y] \rangle=0 , \,\, \forall\,\,\, Y \in \mathfrak{m}_{\mathcal{V}}.$$
Vu que $\mathfrak{m}_{\mathcal{V}}$ est engendré par $$E_{11}+E_{22},\,\,  E_{23} \hbox{ et } E_{13},$$
il vient que $X \in \mathfrak{g}(g)$ si et seulement si $$\left\{
       \begin{array}{ll}
              (1) \langle g, [E_{11}+E_{22},X]  \rangle=0,          &\hbox{} \\
              (2) \langle g, [E_{23},X]  \rangle=0,    & \hbox{} \\
              (3) \langle g, [E_{13},X]  \rangle=0.    & \hbox{} \\

      \end{array}
\right.$$
 On vérifie alors aisément que  $$\mathfrak{g}(g)=\mathbb{K}\, D, $$
avec
\begin{eqnarray}
\begin{split}
D= T- u\,\wedge_{B}e_{1}-\tau\,\,u\,\, \wedge_{B}e_{3}-\frac{\tau}{2\zeta} Z+ \frac{1}{2}\,\tau e_{1}\wedge e_{2}+\frac{1}{2}\,\zeta e_{2}\wedge e_{3}.
\end{split}
\end{eqnarray}
Mais alors, on remarque que $[E_{11}+E_{33}, D]=D$ de sorte que $[\mathfrak{g},\mathfrak{g}(g)]\cap \mathfrak{g}(g) \neq \{0\}$.\\
On suppose maintenant  $\dim E=8$. Ici $\dim F=2$ et $\mathcal{V}=\{ \{0\}=V_{0} \varsubsetneq V_{1} \varsubsetneq V_{2}=V \}$, $\dim V_{1}=1$ et $\dim V_{2}=\dim V=3$. On peut supposer qu'il existe une base $e_{1},e_{2},e_{3}$ de $V$ adaptée à $\mathcal{V}$ et une base orthonormée $u_{1},u_{2}$ de $F$ telles que $$\xi=e_{1}^{\ast}\wedge e_{2}^{\ast} \hbox{ et } n=\xi+u_{1}^{\ast}\wedge e_{3}^{\ast}. $$
Soit $$E^{\prime}=\langle u_{2}\rangle^\bot,\, F^{\prime}=F \cap E^{\prime},\, \mathfrak{n}_{V}^{\prime}=F^{\prime}\wedge_{B}V\oplus \mathfrak{z}_{V},$$
et $\mathcal{V}^{\prime}=\mathcal{V}$ considéré comme drapeau de sous-espaces isotropes de $E^{\prime}$. Alors on a $$ \mathfrak{g}^{\prime}=\mathfrak{p}_{\mathcal{V}^{\prime}}=\mathfrak{q}_{\mathcal{V}}\oplus \mathfrak{n}_{V}^{\prime}\subset \mathfrak{p}_{\mathcal{V}}=\mathfrak{p}_{\mathcal{V}^{\prime}}\oplus\mathfrak{s}\mathfrak{o}(F)\oplus u_{2}\wedge_{B} V.$$
On note $g^{\prime}=g|_{\mathfrak{g}^{\prime}}$ et $n^{\prime}=n|_{\mathfrak{n}_{V}^{\prime}}$. D'après le numéro 5.19 de \cite{DKT}, on a
$$\mathfrak{g}^{\prime}(g^{\prime})\subset \mathfrak{h}:=\mathfrak{g}(n)=\mathfrak{g}^{\prime}(n^{\prime})\oplus \mathbb{K} u_{2}\wedge_{B}e_{3}$$
avec $$\mathfrak{g}^{\prime}(n^{\prime})=\mathbb{K}\, H_{1}\oplus\mathbb{K}\, Z \oplus \mathbb{K}\,(T-u_{1}\wedge_{B} e_{1}) \oplus \mathbb{K}\, u_{1}\wedge_{B}e_{3}\oplus\mathfrak{z}_{V}.$$
On a $$[\mathfrak{s}\mathfrak{o}(F),\mathfrak{g}^{\prime}(n^{\prime})]\subset \mathbb{K}\, u_{2}\wedge_{B} e_{1}\oplus\mathbb{K}\, u_{2}\wedge_{B} e_{3}\subset \ker(n),$$
et $$[u_{2}\wedge_{B} V,\mathfrak{g}^{\prime}(n^{\prime})]\subset u_{2}\wedge_{B} V \subset \ker(n).$$
Comme $\mathfrak{g}^{\prime}(g^{\prime})\subset\mathfrak{g}^{\prime}(n^{\prime})$, on voit alors que $\mathfrak{g}^{\prime}(g^{\prime})\subset\mathfrak{g}(g)$. Ainsi, on a $\mathfrak{g}(g)=\mathfrak{g}^{\prime}(g^{\prime})=\mathbb{K}\, D$. \\
\\Pour le cas général et compte tenu des hypothèses sur la forme linéaire $g$, on peut choisir une base ${e_{1},e_{2},...,e_{2p+1}}$  adaptée au drapeau  $\mathcal{V}$ telle que, notant $e_{1}^{\ast},...,e_{2p+1}^{\ast}$ la base duale, on ait $\xi= \sum_{1\leq i \leq p}e_{2i-1}^{\ast} \wedge e_{2i}^{\ast}$ et $n=\xi+  u^{\ast}\wedge e_{2p+1}^{\ast}$ avec $u \in F$ non isotrope. De plus, on peut supposer que $g|_{\mathfrak{m}_{\mathcal{V}}}=g|_{\mathfrak{s}\mathfrak{o}(F)}=g|_{\mathfrak{l}_{\mathcal{V}}}=0$.\\
\\Considérons, pour $1\leq k \leq p-1$, les vecteurs $C_{k}$. Remarquons tout d'abord que les vecteurs $C_{k}$ appartiennent à $\mathfrak{h}$. De plus, on a $$\langle g, [C_{k},Y]  \rangle=0,\,\forall\,\,Y\,\in \mathfrak{r}_{\mathcal{V}}\hbox{ et } [C_{k},\mathfrak{s}\mathfrak{o}(F)]=0. $$
D'autre part et compte tenu des calculs du cas précédent, on a
$$  \langle g, [C_{k},Y] \rangle=0 ,\,\, \forall \,\, Y \, \in \mathfrak{m}_{\mathcal{V}},$$
 de sorte que $ \langle g, [C_{k},Y] \rangle=0$ pour tout  $Y \, \in \mathfrak{g}$. On a donc $C_{k} \in \mathfrak{g}(g)$.\\
\\En nous inspirant de $(4.7)$ on définit le vecteur $D_{p}$ en posant $$D_{p}=T_{p}- u\wedge_{B} e_{2p-1}-\tau_{p} u\wedge_{B}e_{2p+1}-\frac{\tau_{p}}{2\zeta_{p}}\, Z_{p}+ \frac{1}{2}\tau_{p}\, e_{2p-1}\wedge e_{2p}+\frac{1}{2} \zeta_{p} e_{2p}\wedge e_{2p+1}.$$
Remarquons tout d'abord que $D_{p} \in \mathfrak{h}$ de sorte que
$$ D_{p} \in \mathfrak{g}(g) \hbox{ si et seulement si } \left\{
                                          \begin{array}{ll}
                                            \langle g, [D_{p},Y]  \rangle=0, & \hbox{$ \forall \,\, Y \,\, \in \mathfrak{r}_{\mathcal{V}}$,} \\
                                             \langle g, [D_{p},Y]  \rangle=0, & \hbox{$ \forall \,\, Y \,\, \in \mathfrak{s}\mathfrak{o}(F)$,} \\
                                             \langle g, [D_{p},Y] \rangle=0 , & \hbox{$ \forall \,\, Y\,\,  \in \mathfrak{m}_{\mathcal{V}}$.}
                                          \end{array}
                                        \right.$$
Compte tenu de $(4.4)$, on a $$\langle g,[D_{p},\mathfrak{r}_{\mathcal{V}}]\rangle=0 \hbox{ si et seulement si }
\left\{
  \begin{array}{ll}
    \langle g, [Z_{j},D_{p}]  \rangle=0, & \hbox{  $1\leq j\leq p+1$,} \\
    \langle g, [T_{j},D_{p}]  \rangle=0, & \hbox{  $1\leq j\leq p$,} \\
    \langle g, [H_{j},D_{p}]  \rangle=0, & \hbox{  $1\leq j\leq p$.}
  \end{array}
\right.
$$
On a, pour $1\leq j\leq p+1$,
$$[H_{j}, D_{p}]=\left\{
    \begin{array}{ll}
      T_{p}-\tau_{p}\,u\,\wedge_{B}e_{2p+1}\,\,\textrm{mod} (ker(n))& \hbox{ si $j=p+1$,} \\
       T_{p}-\frac{\tau_{p}}{\zeta_{p}}\, Z_{p}\,\,\textrm{mod} (ker(n)) & \hbox{ si $j=p$,} \\
      0& \hbox{ si $j\leq p-2$.}
    \end{array}
  \right.
$$
De plus, pour $1\leq j\leq p$, on vérifie aisément que $$ [D_{p}, Z_{j}]=\left\{
                                                           \begin{array}{ll}
                                                             0 \,\,\textrm{mod} (ker(n)) & \hbox{ si $j=p$,} \\
                                                             0 & \hbox{ sinon ,}
                                                           \end{array}
                                                         \right.
$$
et
$$ \,\,\,\,\,\,\,\,\,\,\,\,\,\,\,\,\,\,\,\,\,\,\,\,\,\,\,\,\,\,\,\,\,[D_{p}, T_{j}]=\left\{
                                                           \begin{array}{ll}
                                                             0 \,\,\textrm{mod} (ker(n)) & \hbox{ si $j=p$ ou $j=p-1$,} \\
                                                             0 & \hbox{ sinon,}
                                                           \end{array}
                                                         \right.
$$
de sorte que, compte tenu de nos hypothèses sur la forme linéaire $g$,  $$\langle g, [D_{p},Y]  \rangle=0, \hbox{ pour tout }  Y\in \mathfrak{r}_{\mathcal{V}}.$$
\\De plus, il est clair que, pour tout $Y\in \mathfrak{s}\mathfrak{o}(F)$, on a $[Y,D_{p}] \in ker(n)$ de sorte que  $\langle g, [D_{p}, \mathfrak{s}\mathfrak{o}(F)]  \rangle=\{0\}$.\\
\\Par suite, compte tenu de $(4.5)$, on a  $$ D_{p} \in \mathfrak{g}(g) \hbox{ si et seulement si } \left\{
                                                    \begin{array}{ll}
                                                     \,\,\, \langle g, [S_{j}, D_{p}]  \rangle=0, & \hbox{ pour $1\leq j\leq p$,} \\
                                                     \,\,\, \langle g, [E_{2j,2j+1}, D_{p}]  \rangle=0, & \hbox{ pour $1\leq j\leq p$,} \\
                                                     \,\,\, \langle g, [E_{i,j}, D_{p}]  \rangle=0, & \hbox{ pour $2\leq i+1 < j \leq 2p+1 $.}
                                                    \end{array}
                                                  \right.$$
D'une part, on a $$[S_{j}, D_{p}]=\left\{
    \begin{array}{ll}
      -T_{p}+\tau_{p}\,e_{2p-1}\wedge_{B}e_{2p}\,\,\textrm{mod} (ker(n))& \hbox{ si $j=p$,} \\
       0 & \hbox{ si $j\leq p-1$,}
    \end{array}
  \right.
$$
de sorte que, pour $1\leq k\leq p$, $$\langle g, [S_{j}, D_{p}]  \rangle=0.$$
D'autre part, on a
$$[E_{2j,2j+1}, D_{p}]=\left\{
    \begin{array}{ll}
      E_{2p,2p}-E_{2p+1,2p+1}+\frac{\tau_{p}}{2\zeta_{p}}\,E_{2p-1,2p+1}\,\,\textrm{mod} (ker(n))& \hbox{ si $j=p$,}\\
       -\frac{\tau_{p}}{2\zeta_{p}}\,E_{2p-2,2p}\,\textrm{mod} (ker(n)) & \hbox{ si $j=p-1$,} \\
      0 & \hbox{ si $j\leq p-2$.}
    \end{array}
  \right.
$$
Comme $E_{2p,2p}-E_{2p+1,2p+1} \in \mathfrak{l}_{\mathcal{V}}+\mathfrak{m}_{\mathcal{V}}$ et $E_{2p-1,2p+1},E_{2p-2,2p} \in \mathfrak{m}_{\mathcal{V}} $, on voit alors que, pour $1\leq j\leq p$,   $$\langle g, [E_{2j,2j+1}, D_{p}]  \rangle=0.$$
Enfin, pour calculer le crochet $[E_{i,j},D_{p}]$ avec $2\leq i+1 < j \leq r$, on distingue suivant la parité de $i$ et $j$.\\
\\(i) Pour $i=2k^{\prime}$ et $j=2l^{\prime}$ avec $k^{\prime}+1 \leq  l^{\prime} \leq  p$, on a
$$[E_{2k^{\prime},2l^{\prime}}, D_{p}]=\left\{
    \begin{array}{ll}
      -\frac{1}{2}\,\tau_{p}\,\,e_{2k^{\prime}}\wedge e_{2p-1}+\frac{1}{2}\,\,\zeta_{p}\,\,e_{2k^{\prime}}\wedge e_{2p+1}=0\,\,\textrm{mod} (ker(n))& \hbox{ si $l^{\prime}=p$, } \\
       0\,& \hbox{ sinon, }
    \end{array}
  \right.$$
de sorte que $$\langle g, [E_{2k^{\prime},2l^{\prime}}, D_{p}]\rangle=0.$$
(ii) Pour $i=2k^{\prime}$ et $j=2l^{\prime}+1$ avec $1 \leq k^{\prime}<   l^{\prime} \leq  p$, on a
$$[E_{2k^{\prime},2l^{\prime}+1}, D_{p}]=\left\{
    \begin{array}{ll}
      E_{2k^{\prime},2p}-\tau_{p}\,\,u\wedge_{B} e_{2k^{\prime}}-\frac{1}{2}\,\,\zeta_{p}\,\,e_{2k^{\prime}}\wedge e_{2p}& \hbox{ si $l^{\prime}=p$, } \\
       -u\wedge_{B} e_{2k^{\prime}}-\frac{\tau_{p}}{2\zeta_{p}}\,\, E_{2k^{\prime},2p}+\frac{1}{2}\,\tau_{p}\,e_{2k^{\prime}}\wedge e_{2p}& \hbox{ si $l^{\prime}=p-1$, }\\
0  & \hbox{ sinon.}
\end{array}
  \right.$$
Remarquons que $[E_{2k^{\prime},2l^{\prime}+1}, D_{p}]\in \mathfrak{m}_{\mathcal{V}}+\ker(n)$ de sorte que $$\langle g, [E_{2k^{\prime},2l^{\prime}+1}, D_{p}]=0\rangle.$$
(iii) Pour $i=2k^{\prime}+1$ et $j=2l^{\prime}$ avec $1\leq k^{\prime}+1 <  l^{\prime} \leq  p$, on a
$$[E_{2k^{\prime}+1,2l^{\prime}}, D_{p}]=\left\{
    \begin{array}{ll}
      -\frac{1}{2}\,\,\tau_{p}\,\,e_{2k^{\prime}+1}\wedge e_{2p-1}+ \frac{1}{2}\,\,\zeta_{p}\,\,e_{2k^{\prime}+1}\wedge e_{2p+1}& \hbox{ si $l^{\prime}=p$,} \\
0  & \hbox{ sinon.}
\end{array}
  \right.$$
On voit alors que $$[E_{2k^{\prime}+1,2l^{\prime}}, D_{p}]=\left\{
    \begin{array}{ll}
      0 \, \textrm{mod}(\ker(n))& \hbox{ si $l^{\prime}=p$, } \\
0 & \hbox{ sinon, }
\end{array}
  \right.$$
de sorte que $$\langle g, [E_{2k^{\prime}+1,2l^{\prime}}, D_{p}]\rangle=0.$$

(iv) Pour $i=2k^{\prime}+1$ et $j=2l^{\prime}+1$ avec $0 \leq k^{\prime} < l^{\prime} \leq  p$, on a
$$[E_{2k^{\prime}+1,2l^{\prime}+1}, D_{p}]=\left\{
    \begin{array}{ll}
      E_{2k^{\prime}+1,2p}-\tau_{p}\,\,u\wedge_{B} e_{2k^{\prime}+1}-\frac{1}{2}\,\,\zeta_{p}\,\,e_{2k^{\prime}+1}\wedge e_{2p}& \hbox{ si $l^{\prime}=p$, } \\
       -u\wedge_{B} e_{2k^{\prime}+1}-\frac{\tau_{p}}{2\zeta_{p}}\,\, E_{2k^{\prime}+1,2p}+\frac{1}{2}\,\tau_{p}\,e_{2k^{\prime}+1}\wedge e_{2p}& \hbox{ si $l^{\prime}=p-1$, }\\
0& \hbox{ sinon.}
\end{array}
  \right.$$
Par suite, on a les résultats suivants:\\
 \\- Pour $ l^{\prime}< p-1$, on a
$$[E_{2k^{\prime}+1,2l^{\prime}+1}, D_{p}]=0.$$
- Pour $ l^{\prime}=p-1$, on a
$$[E_{2k^{\prime}+1,2l^{\prime}+1}, D_{p}]\in \mathfrak{m}_{\mathcal{V}}+ker(n).$$
- Pour $ l^{\prime}=p$ et $k^{\prime}<p-1$
$$[E_{2k^{\prime}+1,2l^{\prime}+1}, D_{p}] \in \mathfrak{m}_{\mathcal{V}}+ker(n).$$
- Pour $l^{\prime}=p$ et $k^{\prime}=p-1$, on a
\begin{eqnarray*}
\begin{split}
[E_{2k^{\prime}+1,2l^{\prime}+1}, D_{p}]&= E_{2p-1,2p}-\tau_{p}\,\,u\wedge_{B} e_{2p-1}-\frac{1}{2}\,\zeta_{p}\,\,e_{2p-1}\wedge e_{2p}\\
&=\frac{1}{2}\,\,Z_{p}-\frac{1}{2}\,\,\zeta_{p}\,\,e_{2p-1}\wedge e_{2p}\,\,\textrm{mod}(ker(n)).\\
\end{split}
\end{eqnarray*}
Compte tenu  de nos hypothèses, il suit que
$$\langle g, [E_{2k^{\prime}+1,2l^{\prime}+1}, D_{p}]\rangle=0,\,\,\,\,\,0 \leq k^{\prime} < l^{\prime} \leq  p.$$
D'après ce qui précède, on voit que  $$\langle g, [D_{p},Y]  \rangle=0, \hbox{ pour tout }  Y\in \mathfrak{g}.$$
On a donc  $D_{p} \in \mathfrak{g}(g)$.\\
\\Remarquons que la famille $\{C_{k}, 1\leq k \leq p-1, D_{p}\}$ est une famille libre de $\mathfrak{g}(g)$. De plus, comme $\dim(\mathfrak{g}(g))=h(\mathcal{V})=p$, cette famille est maximale et donc elle forme une base de $\mathfrak{g}(g)$. On a donc $$\mathfrak{g}(g)=\bigoplus_{k=1}^{p-1}\mathbb{K} C_{k} \oplus \mathbb{K} D_{p}.$$
On pose $R_{k}=E_{2k-1,2k-1}+E_{2k+1,2k+1}$. On a $R_{k} \in \mathfrak{g}$ pour tout  $1\leq k \leq p$ et on a
$$[R_{k},C_{k}]=C_{k},\,\,1\leq k\leq p-1 \hbox{ et } [R_{p},D_{p}]=D_{p},$$
de sorte  que  $[\mathfrak{g},\mathfrak{g}(g)]\cap \mathfrak{g}(g) \neq \{0\}$. Ceci montre alors que $\mathfrak{g}$ est non stable puisque la forme linéaire $g$ est générique.
\end{proof}
Pour achever la démonstration du théorème, on considère le cas o\`{u} $(q,r)$ vérifie $(4.1)$ et le drapeau $\mathcal{V}$ vérifie la condition $(\ast)$ sans que les sous-espaces $V_{i}$ du drapeau soient tous de dimension impaire pour $1\leq i\leq t-1$. On reprend les notations des cas précédents.\\
Soit donc $\mathfrak{g}=\mathfrak{p}_{\mathcal{V}}$ avec $\mathcal{V}=\{ \{0\}=V_{0} \varsubsetneq V_{1} \varsubsetneq V_{2}\varsubsetneq...\varsubsetneq V_{t-1} \varsubsetneq V_{t}=V\}$ un drapeau de sous-espaces isotropes  vérifiant la condition $(\ast)$ et  $(q,r)$ la condition $(4.3)$. En particulier $\dim F\leq 2$ et si  $\dim F= 2$, $r=\dim V$ est impair.\\
On a $\mathfrak{g}=(\mathfrak{q}_{\mathcal{V}}\times \mathfrak{s}\mathfrak{o}(F))\oplus \mathfrak{n}_{V}$ avec $\mathfrak{n}_{V}=F \wedge_{B} V \oplus \mathfrak{z}_{V}$ et  $\mathfrak{z}_{V}=\Lambda^{2}_{B}V$. Soit $g \in \mathfrak{g}^{\ast}$ générique. On peut supposer que $\xi=g|_{\mathfrak{z}_{V}} \in \Lambda^{2}_{B}V^{\ast}$ est générique relativement à $\mathcal{V}$ et, si $r$ est impair que $n=g|_{\mathfrak{n}_{V}}$ est tel que $n=n_{\xi}+ l \wedge e_{r}^{\ast}$ avec $l \in F^{\ast}$ non isotrope.
Soit $k_{0}=0 < k_{1}<...< k_{s-1} < k_{s}=t$ ceux des indices $j\in {0,...,t}$ tels que $\dim V_{j}$ soit pair  ou $j=t$. Pour $1\leq i \leq s$, on pose $W_{i}=V_{k_{i}}\cap V_{k_{i-1}}^{\perp_{\xi}}$ et on désigne par $\mathcal{V}_{i}$ la trace du drapeau $\mathcal{V}$ sur le sous-espace $W_{i}$. On a $V=\oplus_{i=1}^{s}W_{i}$.\\
\\Soit  $\widetilde{\mathcal{V}}\subset \mathcal{V}$ le sous-drapeau constitué des espaces $V_{k_{i}}$, $0 \leq i \leq s $. On a $$\mathfrak{q}_{\mathcal{V}}=\prod_{i=1}^{s}\mathfrak{q}_{\mathcal{V}_{i}}\oplus\, ^{u}\mathfrak{q}_{\widetilde{\mathcal{V}}}$$
o\`{u} $\mathfrak{q}_{\mathcal{V}_{i}}$ désigne la sous-algèbre parabolique de $\mathfrak{g}\mathfrak{l}(W_{i})$ qui stabilise le drapeau $\mathcal{V}_{i}$ et $^{u}\mathfrak{q}_{\widetilde{\mathcal{V}}}$ le radical unipotent de la sous-algèbre parabolique $\mathfrak{q}_{\widetilde{\mathcal{V}}}$.
Comme $$  \mathfrak{h}+ \mathfrak{n}_{V}=\left\{
                                           \begin{array}{ll}
                                             \mathfrak{r}_{\mathcal{V}}+ \mathfrak{n}_{V} & \hbox{ si r est pair ou $F=0$ } \\
                                             \mathfrak{r}_{\mathcal{V}}^{0}+ \mathfrak{n}_{V} & \hbox{ si r est impair et $F\neq 0$.
}
                                           \end{array}
                                         \right.
$$
et $exp(\mathfrak{n}_{V}(n)).g=g+(\mathfrak{h}+\mathfrak{n}_{V})^{\perp}$, on peut supposer que $g|_{^{u}\mathfrak{q}_{\widetilde{\mathcal{V}}}}=0$.\\
\\On pose $\tilde{s}=s$, si $F=0$ ou $r$ est pair, et $\tilde{s}=s-1$, si $F\neq 0$ et $r$ est impair. On considère les sous-algèbres de Lie de $\mathfrak{g}$, $\mathfrak{g}_{i}=\mathfrak{q}_{\mathcal{V}_{i}}\oplus\mathfrak{z}_{W_{i}}$, pour $1\leq i \leq \tilde{s}$ et, si $F\neq 0$ et $r$ est impair, $\mathfrak{g}_{s}=\mathfrak{q}_{\mathcal{V}_{s}}\times \mathfrak{s}\mathfrak{o}(F) \oplus\mathfrak{n}_{W_{s}}$ (en convenant que $\mathfrak{n}_{W_{s}}=F\wedge_{B}W_{s}\oplus \mathfrak{z}_{W_{s}}$). Enfin, on note $g_{i}=g|_{\mathfrak{g}_{i}}$, $1\leq i \leq s$. Alors, on a
$$\mathfrak{g}=\mathfrak{g}_{i}\oplus(\prod_{j\neq i}\mathfrak{q}_{\mathcal{V}_{j}})\oplus \,^{u}\mathfrak{q}_{\widetilde{\mathcal{V}}} \oplus F\wedge_{B}V \oplus(\oplus_{k< l}W_{k}\wedge_{B}W_{l}) \oplus(\oplus_{j\neq i}\wedge_{B}^{2} W_{j}), \,\,1\leq i \leq \tilde{s},$$
et si $F\neq 0$ et $r$ est impair, on a
$$\mathfrak{g}=\mathfrak{g}_{s}\oplus(\prod_{j< s}\mathfrak{q}_{\mathcal{V}_{j}})\oplus \,^{u}\mathfrak{q}_{\widetilde{\mathcal{V}}} \oplus( \oplus_{j< s} F\wedge_{B}W_{j}) \oplus(\oplus_{k< l}W_{k}\wedge_{B}W_{l}) \oplus(\oplus_{j<s} \wedge_{B}^{2}W_{j}).$$
En regardant les crochets de $\mathfrak{g}_{i}$ avec les différents autres sous-espaces apparaissant dans cette décomposition de $\mathfrak{g}$ en somme directe, on peut montrer que $\mathfrak{g}_{i}(g_{i})\subset \mathfrak{g}(g)$.\\
Compte tenu de nos hypothèses, il existe deux sous-espaces consécutifs du drapeau $\mathcal{V}$ qui sont tous deux de dimension impaire. Ainsi l'une au moins des $\mathfrak{g}_{i}$ n'est pas stable d'après le cas précédent. Comme $\mathfrak{g}_{i}(g_{i})\subset \mathfrak{g}(g)$, il résulte alors du lemme \ref{Tauvel Yu} que $\mathfrak{g}$ n'est pas stable.
\end{proof}
\begin{rema}
On garde les notations précédentes. Comme la forme linéaire $g$ est générique, il existe un ouvert de Zariski $\mathrm{U}$ $G$-invariant contenant $g$. Ainsi et comme
$g \in$ $^{u}\mathfrak{q}_{\widetilde{\mathcal{V}}}^{\bot}$, o\`{u}  $^{u}\mathfrak{q}_{\widetilde{\mathcal{V}}}^{\bot}$ désigne l'orthogonal de $^{u}\mathfrak{q}_{\widetilde{\mathcal{V}}}$ dans $\mathfrak{g}^{\ast}$, $\mathrm{U}\cap\,^{u}\mathfrak{q}_{\widetilde{\mathcal{V}}}^{\bot}$ est un ouvert de Zariski non vide de $^{u}\mathfrak{q}_{\widetilde{\mathcal{V}}}^{\bot}$. Soit $\varphi$ la projection de $^{u}\mathfrak{q}_{\widetilde{\mathcal{V}}}^{\bot}$ dans $\mathfrak{g}_{i}^{\ast}$. Comme $\varphi(\mathrm{U}\cap\,^{u}\mathfrak{q}_{\widetilde{\mathcal{V}}}^{\bot})$ est un ouvert de Zariski non vide de $\mathfrak{g}_{i}^{\ast}$, on voit alors que $g_{i}=g|_{\mathfrak{g}_{i}}$ est générique. Ainsi et compte tenu du cas précédent, on a une formule explicite pour le stabilisateur $\mathfrak{g}_{i}(g_{i})$.\\
Lorsque $\mathfrak{g}$ est de rang nul, en utilisant le résultat qui donne l'indice de $\mathfrak{g}$ et des $\mathfrak{g}_{i}$ dans le théorème $\ref{D.K.T}$, on déduit que $\mathfrak{g}(g)=\oplus_{i}\mathfrak{g}_{i}(g_{i})$. Comme on a déterminé explicitement $\mathfrak{g}_{i}(g_{i})$ dans le cas précédent, on a alors une formule explicite pour le stabilisateur d'une forme linéaire générique.
\end{rema}

\bibliographystyle{smfplain}
\bibliography{biblio}

\end{document}